\documentclass[11pt]{article}
\usepackage[numbers]{natbib}
\usepackage{graphicx,amsmath,epsfig,amssymb,amsfonts,amsthm,enumerate,verbatim,todonotes,hyperref,dsfont}
\usepackage{tikz,tikz-cd}
\usetikzlibrary{matrix,arrows}

\newtheorem{theorem}{Theorem}[section]
\newtheorem{corollary}[theorem]{Corollary}
\newtheorem{lemma}[theorem]{Lemma}

\newtheorem{question}[theorem]{Question}
\newtheorem{claim}[theorem]{Claim}

\theoremstyle{definition}
\newtheorem{definition}[theorem]{Definition}

\newcommand{\restrict}{\mathord{\upharpoonright}}
\newcommand{\concat}{\mathbin{\raisebox{1ex}{\scalebox{.7}{$\frown$}}}}

\begin{document}

\title{Computable analogs of cardinal characteristics: Prediction and Rearrangement}
\author{Iv\'an Ongay-Valverde\footnote{Work done while being supported by CONACYT scholarship for Mexican student studying abroad.}\\\emph{Department of Mathematics}\\\emph{University of Wisconsin--Madison}\\\emph{Email: ongay@math.wisc.edu}\vspace{.4cm}\\Paul Tveite\\\emph{Software Engeenering Department}\\\emph{Google}\\\emph{Email: paul.tveite@google.com}
}
\date{\begin{tabular}{rl}First Draft:&July 11, 2016\\Current Draft:& April 14, 2019\end{tabular}}
\maketitle
\begin{abstract}
There has recently been work by multiple groups in extracting the properties associated with cardinal invariants of the continuum and translating these properties into similar analogous combinatorial properties of computational oracles. Each property yields a highness notion in the Turing degrees. In this paper we study the highness notions that result from the translation of the evasion number and its dual, the prediction number, as well as two versions of the rearrangement number. When translated appropriately, these yield four new highness notions. We will define these new notions, show some of their basic properties and place them in the computability-theoretic version of Cicho\'{n}'s diagram.
\end{abstract}
\section{Introduction}

Recent work of Rupprecht \cite{rupprechtthesis} and, with some influence of Rupprecht but largely independently, Brendle, Brooke-Taylor, Ng, and Nies \cite{bbtnn} developed and showed a process for extracting the combinatorial properties of cardinal characteristics and translating them into highness properties of oracles with related combinatorial properties. Some of the analogs so derived are familiar computability-theoretic properties, some are new characterizations of existing notions, and some are completely new. It is interesting to notice that many of the proofs of relationships between the cardinals in the set-theoretic setting translate to the effective setting. The work so far has mostly focused on the cardinal characteristics of Cicho\'{n}'s diagram. 

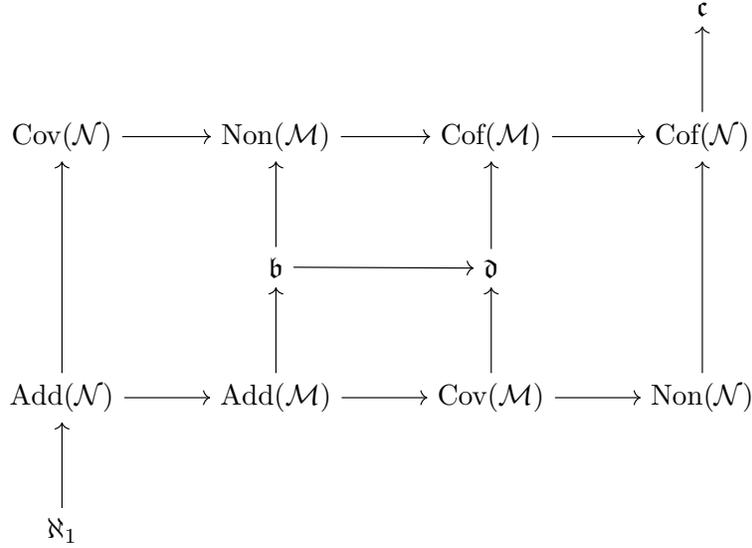
\begin{figure}
\centering
\begin{tikzpicture}
  \matrix (m) [matrix of math nodes,row sep=3em,column sep=3em]
  {
     &&&\mathfrak{c}\\
     \mbox{Cov}(\mathcal{N})&\mbox{Non}(\mathcal{M})&\mbox{Cof}(\mathcal{M})&\mbox{Cof}(\mathcal{N})\\
     &\mathfrak{b}&\mathfrak{d}&\\
     \mbox{Add}(\mathcal{N})&\mbox{Add}(\mathcal{M})&\mbox{Cov}(\mathcal{M})&\mbox{Non}(\mathcal{N})\\
     \aleph_1\\};
  \path[commutative diagrams/.cd, every arrow, every label,font=\scriptsize]
    (m-5-1) edge (m-4-1)
    (m-4-1) edge (m-2-1)
    		edge (m-4-2)
    (m-4-2) edge (m-4-3)
            edge (m-3-2)
    (m-4-3) edge (m-4-4)
    		edge (m-3-3)
    (m-4-4) edge (m-2-4)
    (m-3-2) edge (m-3-3)
    		edge (m-2-2)
    (m-3-3) edge (m-2-3)
    (m-2-1) edge (m-2-2)
    (m-2-2) edge (m-2-3)
    (m-2-3) edge (m-2-4)
    (m-2-4) edge (m-1-4);
\end{tikzpicture}
\caption{Cicho\'{n}'s diagram}
\end{figure}

The nodes in Cicho\'{n}'s diagram, figure 1, have the usual notation and meaning as defined in \cite{kunen2011set}, we will work with most of them in this paper. It is important to notice that the arrows in figure 1 stand for inequalities, with $A\rightarrow B$ in the diagram indicating $A\leq B$.

There is a purely semantic formulation of the translation scheme to an effective notion where all of these characteristics can be viewed as either an unbounding number or a dominating number along the lines of $\mathfrak{b}$ and $\mathfrak{d}$ for a different relationship between two spaces. They can then be semantically converted to the appropriate highness notion. For all the details of the semantic scheme, see \cite{rupprechtthesis} or \cite{bbtnn}. 

An alternative, somewhat intuitive way to think about this translation scheme is to frame it as follows: when working with cardinal characteristics on the set theory side, it is common to build models by forcing extensions that have specific properties, one way to do this is to force a characteristic to be larger by building an extension which has a new object that negates the desired property for a specific collection from the ground model. If we reinterpret the ground model as the computable objects, and the extension as adding those things computable from an oracle, the degree corresponding to the characteristic will be exactly the combinatorial definition needed to negate the characteristic property for the collection of computable objects. Among other things, this means that the highness notions actually end up looking like the negations of the characteristics that they were derived from.

For example, let us take the unbounding number $\mathfrak{b}$. In building a forcing extension to make $\mathfrak{b}$ larger, we would want to add a function which \emph{does} bound a collection of functions from the ground model. When translated to a computability-theoretic highness notion, this becomes an oracle which computes a function dominating every computable function. This is exactly the set of oracles of high degree. Similarly, for the dominating number $\mathfrak{d}$, in building a forcing extension to make $\mathfrak{d}$ larger, we would want to add a function which is not dominated by any function from the ground model. When translated to the computability side, this becomes an oracle which computes a function not dominated by any computable function, i.e.\ of hyperimmune degree. Some of the analogs, like these, are well-studied, and some were introduced by Rupprecht in \cite{rupprechtthesis}. 

\newcommand{\nonn}{Computes a Schnorr Random}
\newcommand{\covm}{Weakly meager engulfing}
\newcommand{\cofm}{Not low for weak 1-gen}
\newcommand{\cofn}{Not low for Schnorr}
\newcommand{\unbound}{High}
\newcommand{\dominating}{Hyper-immune}
\newcommand{\addn}{Schnorr engulfing}
\newcommand{\nonm}{Computes a weak 1-gen}
\newcommand{\addm}{Meager engulfing}
\newcommand{\covn}{Weakly Schnorr engulfing}
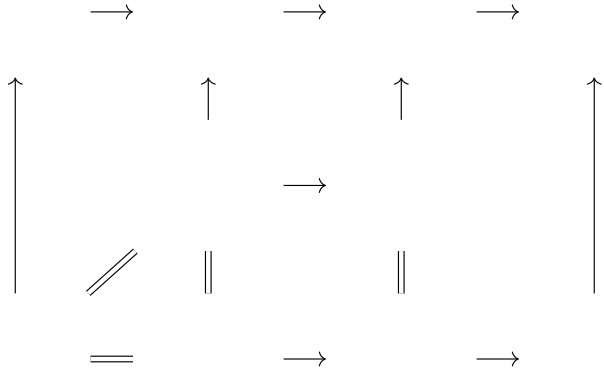
\begin{figure}
\centering
\begin{tikzpicture}[auto,
	every node/.style ={rectangle, fill=white,
      text width=4.5em, text centered, opacity = 0, text opacity = 1, minimum height=4.5em}]
  \matrix (m) [row sep=1.5em,column sep=1.5em]
  {
     \node (11) {\nonn};&\node (21) {\covm};&\node (31) {\cofm};&\node (41) {\cofn};\\
     ;&\node (22) {\unbound};&\node (32) {\dominating};&;\\
     \node (13) {\addn};&\node (23) {\addm};&\node (33) {\nonm};&\node(43) {\covn};\\
     };
	\path[commutative diagrams/.cd, every arrow, every label,font=\scriptsize]
    (11) edge (21)
    (21) edge (31)
    (31) edge (41)
    (22) edge (21)
    	 edge (32)
    (32) edge (31)
    (13) edge (11)
         edge[-,commutative diagrams/equal] (22)
         edge[-,commutative diagrams/equal] (23)
    (23) edge (33)
         edge[-,commutative diagrams/equal] (22)
    (33) edge[-,commutative diagrams/equal] (32)
         edge (43)
    (43) edge (41) ;
\end{tikzpicture}
\caption{Effective Cicho\'{n}'s diagram}
\end{figure}

Figure 2 is a summary of the results known in this area. Here, arrows actually do mean implication, where the lower-left highness properties are generally stronger than the upper-right.
 It is possible to find all the definitions in \cite{rupprechtthesis}, but it is important to remark that some of the Rupprecht terminology is different. 
 
 In this paper, we will expand on this work by looking at four different cardinal characteristics not appearing in Cicho\'{n}'s diagram. First, we will examine the evasion number, a cardinal characteristic first introduced by Blass in \cite{blass1994cardinal}, as well as its less-studied dual, the prediction number. We will also look at two forms of the so-called rearrangement number, as introduced by Blass et al.\ in \cite{blass2016rearrangement}. In all these cases, we will give the correct effective analogs, and prove relationships between these new highness notions and their relationships with other properties which are analogous to well-studied cardinal characteristics.

The questions in this paper were independently studied by Noam Greenberg, Gregory Igusa, Rutger Kuyper, Menachem Magidor and Dan Turetsky. There is significant overlap between their results and those we present below.

We thank the referee for his helpful insights to improve the overall presentation of the paper and helping us improve Theorem \ref{Weaklynotprediction}. Also, we thank Joe Miller for all his advise and help.

\section{Prediction and Evasion}

\subsection{Definitions}

\begin{definition}[Blass \cite{blass1994cardinal}] A \emph{predictor} is a pair $P=(D,\pi)$ such that $D\in[\omega]^{\omega}$ (infinite subsets of $\omega$) and $\pi$ is a sequence $\langle\pi_n:n\in D\rangle$ where each $\pi_n:\omega^n\rightarrow\omega$. By convention, we will sometimes refer to $\pi_n(\sigma)$ by simply $\pi(\sigma)$. This predictor $P$ \emph{predicts} a function $x\in\omega^{\omega}$ if, for all but finitely many $n\in D, \pi_n(x\restrict_n)=x(n)$. Otherwise $x$ \emph{evades} $P$. The \emph{evasion number} $\mathfrak{e}$ is the smallest cardinality of any family $E\subseteq\omega^{\omega}$ such that no single predictor predicts all members of $E$. \end{definition}

We will also make use of the dual to $\mathfrak{e}$, which is explored by Brendle and Shelah in \cite{brendle2}.

\begin{definition}The \emph{prediction number}, which we will call $\mathfrak{v}$ (as Kada in \cite{kada1998baire}), is the smallest cardinality of any family $V$ of predictors such that every function is predicted by a member of $V$.\end{definition}

The known results for $\mathfrak{e}$ and $\mathfrak{v}$ position them as illustrated in figure \ref{cichonwithevasion} relative to Cicho\'{n}'s diagram.

\begin{figure}
\centering
\begin{tikzpicture}
  \matrix (m) [matrix of math nodes,row sep=3em,column sep=1em,minimum width=2em]
  {
     &&&\mbox{Cof}(\mathcal{M})&&\mbox{Cof}(\mathcal{N})\\
     \mbox{Cov}(\mathcal{N})&&\mbox{Non}(\mathcal{M})&&\mathfrak{v}&\\
     &&\mathfrak{b}&\mathfrak{d}&&\\
     &\mathfrak{e}&&\mbox{Cov}(\mathcal{M})&&\mbox{Non}(\mathcal{N})\\
     \mbox{Add}(\mathcal{N})&&\mbox{Add}(\mathcal{M})&&&\\};
  \path[commutative diagrams/.cd, every arrow, every label,font=\scriptsize]
    (m-5-1) edge (m-2-1)
    		edge (m-4-2)
            edge (m-5-3)
    (m-2-1) edge (m-2-3)
    (m-4-2) edge (m-2-3)
    		edge (m-4-4)
    (m-2-3) edge (m-1-4)
    		edge (m-2-5)
    (m-3-3) edge (m-2-3)
    		edge (m-3-4)
    (m-5-3) edge (m-3-3)
    		edge (m-4-4)
    (m-1-4) edge (m-1-6)
    (m-3-4) edge (m-1-4)
    (m-4-4) edge (m-3-4)
    		edge (m-4-6)
    		edge (m-2-5)
    (m-2-5) edge (m-1-6)
    (m-4-6) edge (m-1-6);
\end{tikzpicture}
\caption{Cicho\'{n}'s diagram including $\mathfrak{e}$ and $\mathfrak{v}$.} \label{cichonwithevasion}
\end{figure}

In order to effectivize our prediction-related cardinal characteristics, we must first effectivize the definition of a predictor.

\begin{definition}A \emph{computable predictor} is a pair $P=(D,\langle\pi_n:n\in D\rangle)$ where $D\subseteq\omega$ is infinite and computable, each $\pi_n:\omega^n\rightarrow\omega$ is a computable function and the sequence $\langle\pi_n:n\in D\rangle$ is computable.

Similarly, we define an \emph{$A$-computable predictor} as the relativized version where all objects are computable relative to some oracle $A$. 

Finally, we define an oracle $A$ to be of \emph{evasion degree} if there is a function $f\leq_T A$ which evades all computable predictors, and $A$ is of \emph{prediction degree} if there is a predictor $P\leq_T A$ which predicts all computable functions.
\end{definition}

Because of the fact that we negate the original statements of the definitions of cardinal characteristics, under our scheme the evasion number $\mathfrak{e}$ is an analog to being a prediction degree, and the prediction number $\mathfrak{v}$ is an analog to being an evasion degree.

We present, in Theorem \ref{eofacts}, facts about $\mathfrak{e}$ and $\mathfrak{v}$ represented by Cicho\'{n}'s diagram with $\mathfrak{e}$ and $\mathfrak{v}$ included, as well as their translations into effective analogs. The second table is ordered in a way that stress the duality between the evasion and prediction number.

In the study of cardinal characteristics (in Set Theory), duality plays a central roll through the use of Galois-Tukey connections and the characterization of numbers as $\mathfrak{b}(R)$ or $\mathfrak{d}(R)$ for a relation $R$. In the earlier case, whenever $\mbox{cof}(\mathcal{I})\leq \mbox{cof}(\mathcal{J})$ the Galois-Tukey connection ensures that $\mbox{add}(\mathcal{I})\geq \mbox{add}(\mathcal{J}) $, for the later, whenever $\mathfrak{b}(R)\leq \mathfrak{b}(R')$ then $\mathfrak{d}(R)\geq \mathfrak{d}(R')$. For more on this regard both \cite{bartoszynskibook} and \cite{blass1994cardinal} are good references. In this context, the concept of duality is strong enough to allow researchers to just do proofs for one of the numbers in a dual pair. For example, in \cite{brendle2}, the authors show all results for $\mathfrak{e}$ and state that the analogues are true for $\mathfrak{v}$ due to duality.

On the Computability Theory side, duality still exists through the effectivization of the Galois-Tukey connections (as shown by Rupprecht in \cite{rupprechtthesis}). Nevertheless, some of the highness notions collapses. For example, the analogue of $\mathfrak{d}$ is equal to the analogue of $\mbox{Cov}(\mathcal{M}) $ but the analogue of $\mathfrak{b}$ can be split from the analogue of $\mbox{Non}(\mathcal{M})$. This implies that duality, in this context, is weaker. It seems that it works only on implications (and not in splitings).

This situation causes that some of the results in the Computability Theory side look stronger than the ones from Set Theory (see Theorem \ref{Schnorrpredict}) while there are others that stay the same (see Theorem \ref{predmeagerengulf}). All of these reason made us opt to not cite directly duality, from prediction degrees, in the proofs of the evasion degrees. Notice that Theorem \ref{predmeagerengulf} and \ref{meagerengulfevade} (and their proofs) are clearly dual one to another but the proofs of the pair of Theorems \ref{pred1gen} and \ref{meagerengulfevade}, that should be dual, are not, since the later one involves DNC degrees.

\begin{theorem}\label{eofacts}The following relationships are known for $\mathfrak{e}$.
\small{\center\begin{tabular}{|l|l|c|}
\hline Cardinal Char.&Highness Properties&Theorem\\
\hline $\mbox{add}(\mathcal{N})\leq \mathfrak{e}$ \cite{blass1994cardinal}&Schnorr engulfing $\Rightarrow$ prediction degree&\ref{Schnorrpredict}\\
\hline $\mathfrak{e}\leq\mbox{non}(\mathcal{M})$ \cite{blass1994cardinal}&prediction degree $\Rightarrow$ weakly meager engulfing&\ref{predmeagerengulf}\\
\hline $\mathfrak{e}\leq\mbox{cov}(\mathcal{M})$ \cite{kada1998baire}&prediction degree $\Rightarrow$ weakly 1-generic&\ref{pred1gen}\\
\hline CON($\mathfrak{e}<\mbox{add}(\mathcal{M})$) \cite{brendle1}&meager engulfing $\not\Rightarrow$ prediction degree&False\\
\hline CON($\mathfrak{b}<\mathfrak{e}$) \cite{brendle2}&prediction degree $\not\Rightarrow$ high&\ref{predhighsep}\\
\hline CON($\mathfrak{e}<\mbox{cov}(\mathcal{N})$) \cite{bartoszynskibook}&computes Schnorr Random $\not\Rightarrow$ prediction degree &\ref{predschnorr}\\
\hline
CON($\mbox{cov}(\mathcal{N})<\mathfrak{e}$) \cite{brendle2}&prediction degree $\not\Rightarrow$ computes Schnorr Random & Open\\
\hline
\end{tabular}\\}  
\normalsize Similarly, for $\mathfrak{v}$ (all results can be found in \cite{brendle2}, unless otherwise stated):
\small{\center\begin{tabular}{|l|l|c|}
\hline Cardinal Char.&Highness Properties&Theorem\\
\hline $\mathfrak{v}\leq\mbox{cof}(\mathcal{N})$&evasion degree $\Rightarrow$ not low for Schnorr tests&\ref{evadeschnorr}\\
\hline $\mbox{cov}(\mathcal{M})\leq\mathfrak{v}$&weakly 1-generic $\Rightarrow$ evasion degree&\ref{hyperimmuneevade}\\
\hline $\mbox{non}(\mathcal{M})\leq\mathfrak{v}$ \cite{kada1998baire}&weakly meager engulfing $\Rightarrow$ evasion degree&\ref{meagerengulfevade}\\
\hline CON($\mbox{cof}(\mathcal{M})<\mathfrak{v})$&evasion degree $\not\Rightarrow$ not low for 1-generics&Open\\
\hline CON($\mathfrak{v}<\mathfrak{d}$)&hyperimmune $\not\Rightarrow$ evasion degree&False\\
\hline CON($\mbox{non}(\mathcal{N})<\mathfrak{v}$)\cite{bartoszynskibook}&evasion degree $\not\Rightarrow$ weakly Schnorr engulfing&\ref{Evasion not engulfing}\\
\hline CON($\mathfrak{v}<\mbox{non}(\mathcal{N})$) \cite{bartoszynskibook}&weakly Schnorr Engulfing$\not\Rightarrow$ evasion degree&\ref{Weaklynotprediction}\\
\hline

\end{tabular}}

\

\normalsize These results can be seen in figure \ref{cichonwithevasion} and figure \ref{computablecichonwithevasion}.
\end{theorem}

\begin{figure}
\centering
\begin{tikzpicture}[auto,
	every node/.style ={rectangle, fill=white,
      text width=4.5em, text centered,
      minimum height=2em}]
  \matrix (m) [row sep=1.5em,column sep=1.5em]
  {
     &\node (12) {High or DNC};&\node (13) {Hyper-immune or DNC};&&\node (15) {Not computably traceable};\\
     \node (21) {Computes a Schnorr Random};&\node (22) {Weakly meager engulfing};&\node (23) {Not low for weak 1-gen};&\node (24) {Evasion degree};&\node (25) {Not low for Schnorr tests};\\
     &\node (32) {Prediction degree};&\node (33) {Hyper-immune degree};&&\\
     &\node (42) {High degree};&&&\\
     \node (51) {Schnorr engulfing};&\node (52) {Meager engulfing};&\node (53) {Computes a weak 1-gen};&&\node (55) {Weakly Schnorr engulfing};\\
     };
	\path[commutative diagrams/.cd, every arrow, every label,font=\scriptsize]
    (24) edge[dotted] (13)
    (32) edge[dotted] (21)
    (21) edge (22)
    (22) edge (23)
    (23) edge (24)
    (24) edge (25)
    (32) edge (22)
    	 edge (33)
    (33) edge (23)
    (42) edge (32)
    (53) edge (55)
    (52) edge (53)
    (55) edge (25)
    (51) edge (21)
	(12) edge[-,commutative diagrams/equal] (22)
	(23) edge[-,commutative diagrams/equal] (13)
	(25) edge[-,commutative diagrams/equal] (15)
	(51) edge[-,commutative diagrams/equal] (52)
	(51) edge[-,commutative diagrams/equal] (42)
	(42) edge[-,commutative diagrams/equal] (52)
    (53) edge[-,commutative diagrams/equal] (33);
\end{tikzpicture}
\caption{Effective Cicho\'{n}'s diagram including prediction and evasion degrees. Dotted lines are open questions.} \label{computablecichonwithevasion}
\end{figure}
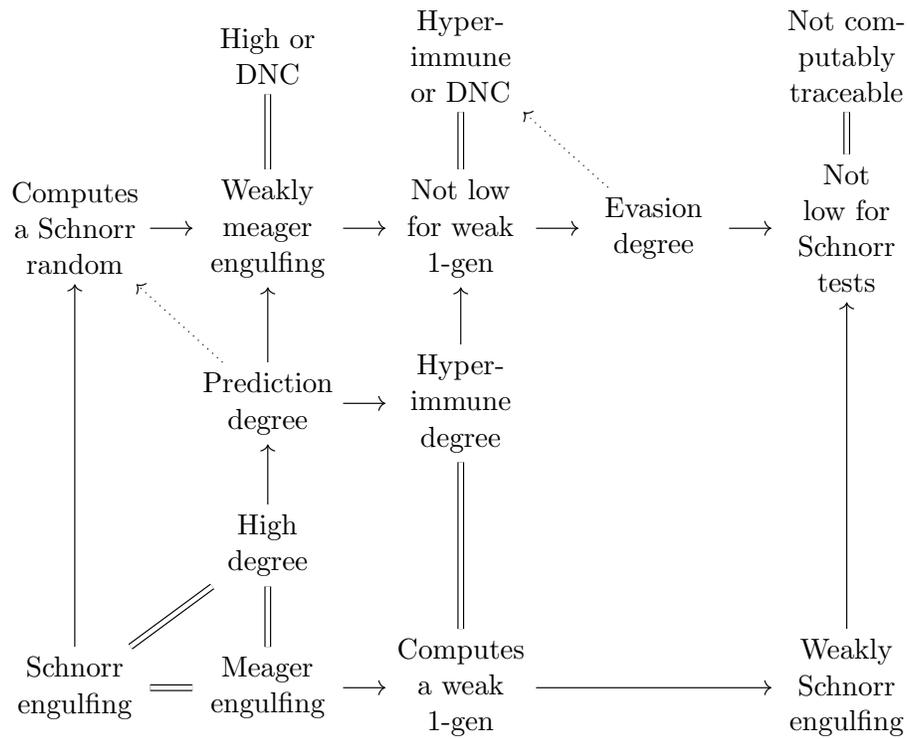

\subsection{Prediction Degrees}

\begin{theorem}\label{Schnorrpredict}If $A\in2^{\omega}$ is high, then it is of prediction degree.\end{theorem}

\begin{proof}Let $A$ be high and set $D=\omega$. We will use the fact that if $A$ is high, then $A$ can enumerate a list of indices for the total computable functions. A proof of this fact can be found in \cite{jockusch1972degrees}. Using this, we simply enumerate all the computable functions. Then to define $\pi_n$, for each finite string $f\in\omega^n$, we go through the list of computable functions $\{\varphi_e\}$ until we find one such that $\varphi_e\restrict n=f$. Then we define $\pi_n(f) = \varphi_e(n)$. This predictor is computable in $A$ and predicts all computable functions. \end{proof}

\begin{lemma}\label{meager} For any predictor $P$, there 
is an effectively-in-$P$ meager set covering all functions predicted by $P$.
\end{lemma}

\begin{proof}The collection 
\[C_i = \{f:|\{n\in D:\pi(f\restrict_n)\neq f(n)\}|<i\}\]
is nowhere dense and $\Pi^0_1$ in $P$, and the collection of functions predicted by $P$ is exactly $\bigcup\limits_{i\in\mathbb{N}}C_i$.
\end{proof}

\begin{theorem}\label{predmeagerengulf}If $A$ is a prediction degree, then $A$ is
weakly meager engulfing.\end{theorem}
\begin{proof}Assume $A$ is a prediction degree, then there is a predictor $P$ computable from $A$ which predicts all computable functions. In particular, we just need a predictor which predicts all $0,1$-valued computable functions. 

Then, by Lemma \ref{meager} one can, using $P$, effectively find a meager set covering every function predicted by $P$. Thus there is an $A$-effectively meager set covering all $0,1$-valued computable functions, and hence covering all computable reals, as desired.
\end{proof}

\begin{theorem}\label{pred1gen}If $A\in2^{\omega}$ is of prediction degree, then $A$ computes a weakly 1-generic.\end{theorem}

By a result of Kurtz \cite{kurtz1983notions}, we know that computing a weakly $1$-generic is equivalent to compute a hyperimmune function. Then, we will actually prove the equivalent statement that if $A$ is a prediction degree, then $A$ has hyperimmune degree.

This is an analog of the characteristic inequality $\mathfrak{e}\leq\mathfrak{d}$. The above theorem is the analog of the strictly stronger cardinal relation $\mathfrak{e}\leq\mbox{cov}(\mathcal{M})$. However, these notions are one of the places where a relationship that is separable in the set-theoretic case collapses in the computability-theoretic analog, so the theorems are equivalent. The proof follows one of Blass from \cite{blass1994cardinal}.
\begin{proof}Given $A\in2^{\omega}$ which is not weakly 1-generic, by a result of Kurtz, $A$ is hyperimmune-free. In particular we will use the fact that for all $f\leq_T A$ with $f:\omega\times\omega\rightarrow\omega$, there is a function $g\leq_T0$ such that $g>f$.

Let $P=(D_P,\{\pi_n\})\leq_T A$ be a predictor, and define $f:\omega\times\omega\rightarrow\omega$ by 
\[f(n,k)=\begin{cases}\max\left\{\pi_n(t):t\in k^n\right\}&\mbox{if }n\in D_P\\0&\mbox{otherwise.}\end{cases}\]
We note that $f\leq_T A$. Then, by assumption, there is a computable function $g$ such that $g(n,k)>f(n,k)$ for all $n,k$. Then we define \[x(n)=g(n,1+\max\{x(p):p<n\}).\] Now, let $n\in D_{P}$ and $k=1+\max\{x(p):p<n\}$. We note that $x\restrict_n$ is of length $n$ and has all values less than $k$, and so is an admissible $t$ from the definition of $f(n,k)$, so $f(n,k)\geq\pi_n(x\restrict_n)$. On the other hand, by definition of $x$ and the choice of $g$, we also have $x(n)= g(n,k)>f(n,k)$. Thus, we have $x(n)>\pi_n(x\restrict_n)$. Since $n$ was arbitrary, it follows that $x$ evades $P$, and so $A$ is not a prediction degree.
\end{proof}





\begin{theorem}\label{predschnorr}There is an $A\in2^{\omega}$ which computes a Schnorr Random but is not of prediction degree.\end{theorem}

\begin{proof}This follows immediately from the fact that there is $A\in2^{\omega}$ which computes a Schnorr Random, but is hyperimmune-free. See, e.g. \cite{bbtnn} \S 4.2 (2).\end{proof}

The next theorem is an effectivization of the proof of the consistency of $\mathfrak{b}<\mathfrak{e}$ done by Brendle and Shelah in \cite{brendle2}. In their forcing, in order to show that all functions in the extension are bounded by a ground model function they rely on a claim that is analogous to our claim \ref{hcomp}. The proof of both claims are complicated due to the fact that giving a name (or Turing fuctional, in our case) and a condition that force it to be a function, the forcing condition has continuum many compatible conditions with it, nevertheless, we want to encode the maximums of all possible minimum values that the function can take in only countably many functions (in our claim, the set of these functions is called a seer). In both cases, it is shown that the essential information of each condition only depends on a finite part of it.

It is important to say that, for our version of the theorem, we face two extra problems. First, the Turing functional that we are working with might not be total. Second, in order to keep using computable information, we need to find extensions with functions that are either computable or hyperimmune-free. To solve both problems, we rely on the hyperimmune-free basis theorem over a specific compact space.

\begin{theorem}\label{predhighsep}There is an $A$ which is of prediction degree but it is not high.\end{theorem}

\begin{proof}We will force with conditions $\langle d,\pi,F\rangle=p\in \mathbb{P}$ where $d\in2^{<\omega}$ is a finite partial function, $\pi = \{\pi^n:n\in d\}$ and $\pi^n:\omega^n\rightarrow\omega$ is a finite partial function, $F\subset\omega^{\omega}$ is a finite collection of functions with the property $f,g\in F, f\neq g\Rightarrow f\restrict_{|d|}\neq g\restrict_{|d|}$. Here, the $d$ and $\pi$ can be thought of as partial approximations of $D$ and $\pi$ in the eventual predictor we are constructing, and $F$ as the collection of functions that we are committed to predicting correctly for the rest of the construction.

We define $(d',\pi',F')$ as an extension of $(d,\pi,F)$ by
\begin{align*}(d',\pi',F')\leq(d,\pi,F)\iff&d'\supset d, \pi'\supset\pi,F'\supset F\mbox{ and}\\&f\in F,n\in\mbox{dom}(d')\setminus\mbox{dom}(d)\Rightarrow \pi^{\prime}(f\restrict_n)=f(n).\end{align*}

Due to the use of various indexes it is important to make a comment on notation:
\begin{itemize}
\item We will identify $d$ with a finite subset of $\omega$, so $m\in d$ means that $d(m)=1$. Formally, we are using the function since it is important that we are able to decide both the elements that are in and outside of $d$. 

\item Given $q$ a forcing condition, we will express it as $q=\langle ^{q}d, ^{q}\pi, ^{q}F\rangle$, unless otherwise stated.

\item We will do a construction by stages, so the condition that is selected at each stage $s$ will be $p_{s}=\langle d_{s}, \pi_{s}, F_{s} \rangle$.

\item $\pi^{m}$, as mentioned above, will denote the function of $\pi$  corresponding to $m\in d$ . We will not use the superscript for anything else. Also, if there is no confusion, we will denote $\pi^{m}$ as $\pi$.

\item The left subscript will only be use in case we need to enumerate something, in that case $_{i}q=\langle _{i}d,\, _{i}\pi,\, _{i}F \rangle $.
\end{itemize}

To initialize the construction, we let $d_0=\langle\rangle,\pi_0=\{\},F_0=\{\}$. We will maintain the property that the joint $\bigoplus F_s = \bigoplus\limits_{f\in F_s}f$ is hyperimmune-free and we will extend by the following rules:

$Q_e$: The goal of this requirement will be to ensure that we predict $\varphi_e$.

At stage $s=3e$, we simply set $F_s=F_{s-1}\cup\{\varphi_e\}$ and $d_s=d_{s-1}\concat 0^n$ with $n$ least such that for $f\in F_{s-1}$, if $\varphi_e\neq f$, then $\varphi_e\restrict_{|d_s|}\neq f\restrict_{|d_s|}$. Additionally, if $m\in d_{s-1}\subseteq d_{s}$ and $\pi_{s-1}(\varphi_e\restrict_m)$ is undefined, we define  $\pi_{s}(\varphi_e\restrict_m)$ to be $\varphi_e(m)$.

$I_e$: The goal of this requirement is to ensure that $D$ is  infinite.

At stage $s=3e+1$, $d_s=d_{s-1}\concat1$, and $\pi_s=\pi_{s-1}\cup\{\pi^m\}$ with $m=|d_{s-1}|$ where $\pi^m:\omega^m\rightarrow\omega$ with $\pi^m(f\restrict m)=f(m)$ for all $f\in F_{s-1}$, $\pi^m(\sigma)$ undefined for all other $\sigma$, and $F_s=F_{s-1}$.

$E_{e,n}$: The goal of this requirement is to avoid that, once we finish the construction, the predictor $P= \langle\bigcup d_s,\bigcup\pi_s\rangle$ is of high degree. To do this, we will ensure that either $\varphi_e^P$ is not total or that there is a computable function $h^{e}$ such that $\exists^{\infty}n(\varphi_e^A(n)\leq h^{e}(n))$.

In order to create the function $h^{e}(n)$, we have to make a guess depending on every forcing extension below $p_{s}$. Because of that we define:

\begin{definition}
A collection of functions $S$ indexed by $d,\pi,\overline{f}^{\ast}$, where $d,\pi$ are as in $\mathbb{P}$ and $\overline{f}^{\ast}$ is a finite sequence of finite initial segments of functions, is a \emph{seer for $\varphi_{e}$ at stage s} if and only if for any collection $\hat{F}$ of total functions extending  $\overline{f}^{\ast}$, the forcing condition $\langle d,\pi,F_s\cup \hat{F}\rangle$ can be extended by $q=\langle ^{q}d,\, ^{q}\pi,\, ^{q}F\rangle$ such that $\varphi_{e}^{\langle ^{q}d,\, ^{q}\pi\rangle}$ is bounded by an specific function in $S$. Syntactically, the set
\small{\begin{align*}S=\left\{h^e_{d,\pi,\overline{f}^{\ast}}\in\omega^\omega:
\begin{array}{c}
\langle d,\pi,F_s\rangle\leq\langle d_s,\pi_s,F_s\rangle,
\overline{f}^{\ast}=\langle f_{i}^{\ast}\rangle, |\overline{f}^{\ast}|=l\in\omega,f_i^{\ast}\in\omega^{|d|}\\
\mbox{ are distinct and }\forall f\in F_s,f_i^{\ast}\in\overline{f}^{\ast}, f_i^{\ast}\neq f\restrict_{|d|}
\end{array}
\right\}\end{align*}}
\normalsize is a \emph{seer for $\varphi_{e}$ at stage $s$} if and only if for all its elements and for all $n\in \omega$
\begin{align*}h^{e}_{d,\pi,\overline{f}^{\ast}}(n)\geq\min\{m&:\forall p=\langle d,\pi,F_{s}\cup\hat{F}\rangle\mbox{ with }\hat{F}=\{f_i\in\omega^{\omega}\}_{i<l}\mbox{ and }\\
&\ \ f_{i}\restrict_{|d|}=f_{i}^{\ast}\ (\exists q\leq p \ \ \varphi_e^{\langle ^{q}d,\, ^{q}\pi\rangle}(n)\downarrow<m)\}.\end{align*}
\end{definition}

\begin{claim}\label{hcomp}We claim that either\begin{enumerate}[(1)]
\item There is $n\in \omega$ and $p\leq \langle d_{s}, \pi_{s}, F_{s}\rangle$ such that for any $q\leq p$, $\varphi_e^{\langle ^{q}d,\, ^{q}\pi\rangle}(n)\uparrow$ and $\bigoplus\, ^{p}F$ is hyperimmune free

or,

\item There is a uniformly $\Sigma^{0,\bigoplus F_s}_1$  collection of functions indexed by $d,\pi,\overline{f}^{\ast}$ that is a seer for $\varphi_{e}$ at stage $s$.

\end{enumerate}\end{claim}

At stage $s=3\langle e,0\rangle+2$, we will use the above claim in the following way:

If (1), then we define $\langle d_{s+1}, \pi_{s+1}, F_{s+1}\rangle$ to be such a $p$ and we do nothing for stages of the form $s=3\langle e,n+1\rangle+2$. This will make $\varphi_{e}^{\langle d,\pi\rangle}$ not total.

If (2), then we can find a seer for $\varphi_{e}$ at stage $3\langle e,0\rangle+2$. This set is countable and it is indexed in a computable way, then we can find $\widehat{h}^e\leq_T\bigoplus F_{3\langle e,0\rangle+2}$ such that $h^e_{d,\pi,\overline{f}^{\ast}}\leq^{\ast}\widehat{h}^e$ for all such functions. However, since $\bigoplus F_{3\langle e,0\rangle+2}$ is hyperimmune-free, it follows that there is a computable function $h^e$ for which $(\forall n)h^e(n)\geq\widehat{h}^e(n)$. We then resume the construction.

At stage $s=3\langle e,n+1 \rangle+2$ we can find $j> n$ so that $h^{e}(j)\geq h^e_{d_{s}, \pi_{s}, \overline{f}^{\ast}}(j)$ where $\overline{f}^{\ast}$ are the restrictions of the functions in $F_{s}\setminus F_{3\langle e,0\rangle+2}$ to $|d_{s}|$ and such that $\varphi_e^{\langle d_s,\pi_s\rangle}(j)$ is not yet defined.

In this situation, we can find $p_{s+1}=\langle d_{s+1}, \pi_{s+1}, F_{s+1}\rangle$ such that \[\varphi_e^{\langle d_{s+1},\pi_{s+1}\rangle}(j)\downarrow \leq h^e_{d_{s}, \pi_{s}, \overline{f}^{\ast}}(j)\leq h^{e}(j).\]  $\bigoplus F_{s+1}$ may not be hyperimmune-free, however, notice that the convergance of $\varphi_e^{\langle d_{s+1},\pi_{s+1}\rangle}(j)$ only depends on finite initial segments of the members of $F_{s+1}\setminus F_s$, and so there actually is such a condition with $\bigoplus F_{s+1}$ hyperimmune-free. We pick a condition with this property.

Verification: By construction, the predictor $P=\langle\bigcup d_s,\bigcup\pi_s\rangle$ has the desired properties. $Q_e$ ensures our predictor predicts all computable functions, $I_e$ ensures that $\bigcup d_s$ is infinite, and $E_{e,n}$ ensures that the computational strength of the predictor cannot compute a total function dominating the computable functions, specifically, $h^{e}\not\leq^{\ast}\varphi_{e}^{P}$, so $P$ is not high.\end{proof}

Proof of Claim \ref{hcomp}:

\begin{proof}

Before doing the technical work to show the claim, we will explain the idea of the upcoming proof. As we see above, we want -- if possible -- to define the function $h^{e}_{d,\pi,\overline{g}^{\ast}}$ in such a way that, given $\langle d,\pi, F_{s}\cup G\rangle \leq \langle d_s,\pi_s,F_s\rangle$ with $G=\{g_i: i<l+1\}$, $\overline{g}^{\ast}=\langle g_{i}^{\ast}:i<l+1\rangle $, $g_{i}\restrict_{|d|}=g_{i}^{\ast}$ for all $i<l+1$ and $n\in \omega$ we can find $q\leq \langle d,\pi, F_{s}\cup G\rangle$ such that $\varphi_{e}^{\langle ^{q}d, ^{q}\pi\rangle}(n)$ is smaller than $h^{e}_{d,\pi,\overline{g}^{\ast}}(n)$. In other words, $h^{e}_{d,\pi,\overline{g}^{\ast}}(n)$ represents the minimal value that we can force $\varphi_{e}^{\langle D, \pi\rangle}(n)$ to take given that we already committed to $d,\pi,\overline{g}^{\ast}$.

In order to do this, we try to find all the possible extensions $q$ of the node $\langle d,\pi, F_{s}\cup G\rangle$ that make $\varphi_{e}^{\langle ^{q}d,\, ^{q}\pi\rangle}$ take the smallest possible value. There may be infinitely many of these nodes, which is problematic since, a priori, there might not be a computable way (from $\bigoplus F_s$) to find this minimun, nevertheless, we want to keep the function computable from $\bigoplus F_{s} $. To find only finitely many possible nodes in a $\bigoplus F_s$ computable way, we will restrict our possible $G$ to a $\bigoplus F_{s}$ computable compact subspace of $\omega^{\omega}$. Notice that making $\varphi_{e}^{\langle ^{q}d,\, ^{q}\pi\rangle}(n)$ converge is an open condition (it is downwards close in the forcing and, as we will see, open in the compact space), so, if these open sets cover the compact subspace, a finite $\bigoplus F_{s}$ computable subcover will give us only finitely many extensions to consider. Furthermore, in compact spaces we can use the hyperimmune-free basis theorem.

The conversion from the whole $\omega^{\omega}$ to a compact space is possible thanks to the following observation about compatibility: for $ q=\langle ^{q}d,\, ^{q}\pi,\, ^{q}F\rangle\leq \langle d,\pi, F\rangle $ and $\langle d,\pi, \{g\}\rangle$ to be compatible it is sufficient (but not necessary) that $g\restrict|d|$ is different from $f\restrict|d|$ for all $f\in F$ and that there is $t< |d|$ such that $g(t)$ is bigger than the $t$th index of all strings in the domain of any function in $^{q}\pi$ (more formally, it is bigger than $\sigma(t)$ for all $\sigma\in dom( ^{q}\pi)$).\footnote{The compatibility is true because $g\restrict t$ is not defined in the domain of $^{q}\pi$, therefore, we can create $\pi'$ which always predicts $g$ correctly after $t$ such that $^{q}\pi\subseteq \pi'$. In this way $\langle ^{q}d, \pi', F\cup \{g\}\rangle$ is below $q$ and $\langle d,\pi, \{g\}\rangle$.} In particular, this observation hints at the possibility of only worrying about functions of certain growth while we are looking for our small convergences.

Returning to more technical work, during our proof, we will ask $h^{e}_{d,\pi,g_{i}^{\ast}}(n)$ to not only be bigger than $\varphi_{e}^{^{q}d,\, ^{q}\pi}(n)$ for a selected $q$, but also to be bigger than the values taken by strings in the domain of functions from $^{q}\pi$. In that way, we make $h^{e}_{d,\pi,g_{i}^{\ast}}(n)$ carry some information of compatibility with those $q$. To define the compact space where we will work, we will define functions $B_{l}$ that combine nicely the information needed.

Finally, we would like to warn the reader about the information inside the $\max$ functions in the proof. We decide to write explicitly the elements that give the properties. Under a more detailed analysis, some of these requirements may be repeated information. Nevertheless, we hope that this decision improves the flow of the reading.

Now, for the proof, we will show this by induction on $l=|\overline{f}^{\ast}|$. Our induction hypothesis is slightly stronger than the statement of the claim. Case $(1)$ remains unchanged, but we add to case $(2)$ the additional requirement: \begin{enumerate}
\item[(2a)]For all $n, l\in \omega$ there is a function $B_{l}^{n}:\omega\rightarrow\omega$ such that for all $f\in\omega^\omega$ for which there is $j$ such that $f(j)>h^e_{d,\pi,\overline{f}^{\ast}}(n)$ with $j<|d|$, there is $r$ such that $\varphi_{e}^{\langle ^{r}d,\, ^{r}\pi\rangle}(n)\downarrow<h^e_{d,\pi,\overline{f}^{\ast}}(n)$ and such that it is compatible with $\langle d,\pi,F_s\cup\{f\}\cup\widehat{F}\rangle$ with  $\widehat{F}=\{f_i\in\omega^{\omega}\}_{i<l}$, $f_i\restrict|d|=f_{i}^{\ast}$ and $f_{i}(s)\leq B_{l}^{n}(s)$ for all $s\geq |d|$, $i<l$. Furthermore, $r$ do not depends on $f$ (but may depend on $j$, $\hat{F}$ and $n$).\end{enumerate}

\ 

\noindent\textbf{Case $l=0$:}

If there is $n\in \omega$ and $\langle d, \pi, F_{s}\rangle \leq \langle d_{s}, \pi_{s}, F_{s}\rangle$ such that for all $q$ extending $ \langle d, \pi, F_{s}\rangle$ we have that $\varphi_{e}^{\langle ^{q}d,\, ^{q}\pi\rangle}(n)$ diverges then $\langle d, \pi, F_{s}\rangle$ satisfies $(1)$. Otherwise, fix $\langle d,\pi, F_{s}\rangle \leq \langle d_s,\pi_s,F_s\rangle$. We will define a function $h^{e}_{d, \pi, \emptyset}$ computable from $\bigoplus F_{s}$ with the desired properties.

Fix $n\in \omega$. We begin searching for extensions $q\leq\langle d,\pi,F_s\rangle$ with $\varphi^{\langle ^{q}d,\, ^{q}\pi\rangle}_{e}(n)\downarrow$. As soon as we find a convergence to a value $m$, we let 
\[h^{e}_{d, \pi, \emptyset}(n)=\max\{m+1,\min\{k:\forall i\in\, ^{q}d\ \forall\sigma\in\mbox{dom}( ^{q}\pi^{i})\ (\sigma\in k^i)\}+1\}.\]

Notice that the first part of the max ensures (2), and the second part, using $B^{n}_{0}$ the constant function $0$, ensures that (2a) is satisfied: assume that $f(j)>h^{e}_{d, \pi, \emptyset}(n)$, $j< |d|$. The only way there is no extension as described in (2a) is if $^{q}\pi$ incorrectly predicts $f$ for some $\ell\in[|d|,|^{q}d|)$. This is impossible since $f(j)$ takes a value which is larger than anything that shows up in the domain of any of the functions from $^{q}\pi$. Then $^{q}\pi$ do not make any prediction for $f$ in $[|d|,|^{q}d|)$, which shows that $q$ and $\langle d,\pi,F_s\cup\{f\}\rangle$ are compatible. Furthermore, this $q$ do not depend on $f$.

\noindent\textbf{Case $l=1$:}

If (1) has already happened, we are done. Otherwise, fix $\langle d,\pi, F_{s}\rangle$ extending $\langle d_s,\pi_s,F_s\rangle$ and $g^{\ast}\in\omega^{|d|}$ such that for all $f\in F_{s}$, $f\restrict_{|d|}\neq g^{\ast}$.

We will define a function $h^{e}_{d, \pi, \langle g^{*}\rangle}$ computable from $\bigoplus F_{s}$ with the desired properties.

Fix $n\in \omega$. Let $h^e_{d,\pi,\emptyset}$ be as in the $l = 0$ case. We define 

\[B_{1}^n (j)=\left\{\begin{array}{lc}
0 & j<|d|\\
\max\left\{\begin{array}{l}B_1^n(j-1),\\ h^{e}_{d', \pi', \emptyset}(n)\end{array}:\begin{array}{l} \langle d',\pi', F_{s}\rangle\leq \langle d,\pi, F_{s}\rangle,\\ d'=d0^{j-|d|}, \pi'=\pi\end{array}\right\} & |d|\leq j
\end{array}\right. .\]

Since $B_{1}^n$ is computable from $\bigoplus F_{s}$ we have that the space 
\[C_{1}^n=\{ f \in\omega^{\omega} \ : \ f\restrict|d|=g^{\ast} \ \& \ \forall j\geq |d|\ f(j)\leq B_{1}^n(j)\}\] 
is effectively compact with respect to $\bigoplus F_{s}$.

Now, we can define open sets in $C_1^n$ representing bounded convergence. We define these sets as
\[U^{n}_{m}=\{h\in C_1^n \ : \ \exists q\leq \langle d, \pi, F_{s}\cup\{ h\}\rangle \ \varphi^{\langle ^{q}d, ^{q}\pi\rangle}_{e} (n)\downarrow < m \}.\]
Notice that $U^{n}_{m}\subseteq U^{n}_{t}$ as long as $m\leq t$, and that $U^{n}_{m}$ is a $\Sigma_1^{0,\bigoplus F_s}$ set of functions.

Furthermore, if we call $A^{n}=\bigcup\limits_{m\in \omega}U^{n}_{m}$, we have that $C_1^n\setminus A^{n}$ is a $\Pi_1^{0,\bigoplus F_s}$ class that can be expressed as follows:
\[\{h\in C_1^n:\forall q\leq \langle d, \pi, F_{s}\cup\{ h\}\rangle \ \varphi^{\langle ^{q}d, ^{q}\pi\rangle}_{e} (n)\uparrow\}.\]

If $C_1^n\setminus A^{n}\neq \emptyset$, using the hyperimmune-free basis theorem\footnote{This theorem is true in (computable) compact spaces but fails in spaces like $\omega^{\omega}$. } over $C_1^n$, we can find an $h\in C_1^n\setminus A^{n}$ which is hyperimmune-free relative to $\bigoplus F_s$, but since this join is hyperimmune-free, it follows that $h$ is hyperimmune-free, and we can satisfy (1) with  $p=\langle d, \pi, F_{s}\cup\{ h\}\rangle$.

Otherwise, $C_1^n=A^{n}=\bigcup_{m\in \omega}U^{n}_{m}$ so, by compactness, there is $m^{\ast}$, which can be found in an effective way from $\bigoplus F_{s}$, such that $C_1^n=U^{n}_{m^{\ast}}$. This $m^{\ast}$ will help us satisfy (2) within $C^{n}_1$. Now, notice that the set of functions in $C^n_1$ that can be added to $F_s$ and have a small convergence using $d',\pi'$, i.e. 
\[O_{d',\pi'}=\{f\in C_{1}^n : \exists \langle d', \pi', F\rangle \leq \langle d, \pi, F_{s}\cup \{f\}\rangle\ \varphi_{e}^{\langle d', \pi'\rangle} (n)\downarrow<m^{\ast}\},\] is open. This set is $\Sigma_1^{0,\bigoplus F_s}$ and we have that 

\[C_{1}^n=U^{n}_{m^{\ast}}=\bigcup_{\langle d', \pi', \emptyset\rangle\in \mathbb{P}}O_{d',\pi'}.\]

By effective compactness we can find $\alpha(1)=\alpha\in \omega$ and  $\langle _{\xi}d,\, _{\xi}\pi\rangle$  for all $\xi\leq \alpha$, such that $C_{1}^n=\bigcup\limits_{\xi=1}^{\alpha}O_{_{\xi}d,\, _{\xi}\pi}$. In other words, this gives us finitely many $\langle _{\xi}d,\, _{\xi}\pi, F_s\cup\emptyset \rangle$ such that for all $f\in C^n_1$ there is $\xi$, $ \langle _{\xi}d,\, _{\xi}\pi, F_s\cup\{f\} \rangle\leq \langle d, \pi, F_{s}\cup \{f\}\rangle$, forcing a convergence less than $m^{\ast}$.

Let 
\small{\[h^{e}_{d, \pi, \langle g^{\ast}\rangle}(n)=\max\left\{\begin{array}{l}m^{\ast}, \max\{B_{1}^n(j):\exists \xi\leq \alpha (j\leq |_{\xi}d |)\}\\\min\{k:\forall \xi<\alpha\forall i\in{} _{\xi}d \forall \sigma\in \mbox{dom}(_{\xi}\pi^{i}) (\sigma\in k^{i}) \}\end{array}\right\}.\]}

\normalsize

We will verify that this function satisfy (2) and (2a).

For (2), given $f\in C_{1}^n$ there is $\xi \leq \alpha$ such that $\langle _{\xi}d,\, _{\xi}\pi, F_s\cup\{f\} \rangle\leq\langle d, \pi, F_{s}\cup \{f\}\rangle$  and $ \varphi_{e}^{\langle _{\xi}d,\, _{\xi}\pi\rangle} (n)\downarrow<m^{\ast}\leq h^{e}_{d, \pi, \langle g^{\ast}\rangle}(n)$. If $f\notin C_1^n$ and $j$ is the first such that  $f(j)> B_1^n(j)$ we have two cases.

In one case, $j\geq |_{\xi}d|$ for all $\xi<\alpha$. In this situation, we have that $\langle d,\pi,F_s\cup \{f\}\rangle$ is compatible with any extension $r$ of $\langle d,\pi,F_s\cup \{f\restrict j^\frown 0^{\infty}\}\rangle$ with $|^{r}d|<j$. Since $f\restrict j^\frown 0^{\infty}\in C^n_1 $, there is some $\xi\leq \alpha$ such that $\langle _{\xi}d,\, _{\xi}\pi, F_s\cup\{f\restrict j^\frown 0^{\infty}\} \rangle\leq \langle d,\, \pi, F_s\cup\{f\restrict j^\frown 0^{\infty}\} \rangle$ with $|_{\xi}d|<j$, so we are done. For the second case, if there is $\xi\leq \alpha$ such that $j<|_{\xi}d|$, then $f(j)> B_1^n(j)> h^{e}_{d', \pi', \emptyset}(n)$ with $\langle d',\pi', F_{s}\rangle\leq \langle d,\pi, F_{s}\rangle$ where $d'=d0^{|_{\xi}d|-|d|}$ (i.e., $d$ follow by a lot of zeros) and  $\pi'=\pi$. Notice that $|d'|=|_{\xi}d|$ and that $\langle d',\pi', F_{s}\cup\{f\}\rangle$ is an extension of $\langle d,\pi,F_s\cup \{f\}\rangle$. Now, from the analysis at the end of case $l=0$ (i.e., using induction hypothesis with property (2a)), we know that there is $r$ compatible with $\langle d',\pi',F_s\cup \{f\}\rangle$ (and then, also with $\langle d,\pi,F_s\cup \{f\}\rangle$) with \[\varphi_{e}^{\langle ^{r}d,\, ^{r}\pi\rangle}(n)\downarrow<h^e_{d',\pi',\emptyset}(n)<B_1^n(|d'|)<B_1^n(|j|)\leq h^{e}_{d, \pi, \langle g^{\ast}\rangle}(n).\]

For (2a), we use the same reasoning as in the case $l=0$ to show that if $f(j)\geq h^{e}_{d, \pi, \langle g^{\ast}\rangle}(n)$, $j\leq |d|$ then $\langle d,\, \pi, F_s\cup\{f\} \rangle$ is compatible with $\langle _{\xi}d,\, _{\xi}\pi, F_s\emptyset \rangle$ for any $\xi\leq \alpha$. Now, we know that given $g\in C^{n}_{1}$, i.e., such that $g(s)\leq B_{1}^{n}(s)$ for $s\geq |d|$, then $\langle d,\, \pi, F_s\cup\{g\} \rangle$ is extended by some $\langle _{\xi}d,\, _{\xi}\pi, F_s\cup\{g\} \rangle=r$. Therefore, $\langle d,\, \pi, F_s\cup\{g\}\cup\{f\} \rangle$ is compatible with that $r$.

\noindent\textbf{Case $l+1$:}

Fix $\langle d,\pi, F_{s}\rangle \leq \langle d_s,\pi_s,F_s\rangle$ and $g_{i}^{\ast}\in\omega^{|d|}$ for $i\in\{0,\dots,l\}$ such that for all $f\in F_{s}$, and all $i<l+1$, $f\restrict|d|\neq g_{i}^{\ast}$ and $g_{i}^{\ast}\neq g_{j}^{\ast}$ if $i\neq j$. Then, by our inductive hypothesis, we have that for all $A\subset\overline{g}^{\ast}$ with $|A|\leq l$, either case (2) and (2a) hold or case (1) holds. If for any such subset, we see that (1) holds, then by definition, (1) holds  of $\overline{f}^{\ast}$, and we are done. Otherwise, we will define a function $h^{e}_{d, \pi, \langle g_{i}^{*}: i<l+1\rangle}$ computable from $\bigoplus F_{s}$ with the desired properties.

Fix $n\in \omega$. Define 
\small{\[B_{l+1}^n(j)=\begin{cases}
0 & j<|d|\\
           \max\left\{\begin{array}{l}
           B_{l+1}^n(j-1),\\ B_k^n(j),\\
           h^{e}_{d', \pi', \langle f_{i}^{*}\restrict |d'|:i<k\rangle}(n)
                \end{array}
                :
                \begin{array}{l}
                d'=d0^{j-|d|}, \pi'=\pi\\
                |\{ f_{i}^{*}\restrict|d'|:i<k\}|=k<l+1,\\
                (\forall i)f_i^{\ast}\restrict|d|\in\overline{g}^{\ast}\\
                (\forall i)(\forall |d|\leq t< j) f_{i}(t)\leq B_{l+1}^n(t)
                \end{array}\right\}                 & |d|\leq j.
\end{cases}\]}
\normalsize In order for our proof to work, following the idea of case $l=1$, we will define a compact space in $(\omega^{\omega})^{l+1}$ such that each coordinate is bounded by $B_{l+1}^n$. Restricting to the functions in this compact space is sufficient as we will see in the verification at the end of the case. 

Since $B_{l+1}^n$ is computable from $\bigoplus F_{s}$ we have that the space of collections of functions agreeing with $g_i^{\ast}$ up to $|d|$ and bounded by $B_{l+1}^n$ thereafter, defined by\[C_{l+1}^n=\{ \langle f_{i} : i<l+1\rangle: f_{i} \in\omega^{\omega}, f_{i}\restrict|d|=g_{i}^{\ast} \ \& \forall j\geq |d|\ f_{i}(j)\leq B_{l+1}^n(j) \}\] is effectivly compact with respect to $\bigoplus F_{s}$.

Furthermore, we define the sets \small{\[U^n_{m}=\{\langle h_{i} : i<l+1\rangle\in C_{l+1}^n \ : \ \exists q\leq \langle d, \pi, F_{s}\cup\{ h_{i} : i<l+1\}\rangle \ \varphi^{\langle ^{q}d,\, ^{q}\pi\rangle}_{e} (n)\downarrow < m \}.\]}

\normalsize We can do the same as in the case $l=1$. If the compact space is not the union of $U_{m}^{n}$ then we can satisfy (1). Otherwise, define $O_{d',\pi'}$ as before. Then we do as we did for $l=1$ to define  $m^{\ast}$ and $\langle _{\xi}d,\, _{\xi}\pi, F_s\cup\emptyset \rangle$ for all $\xi\leq \alpha(l+1)=\alpha$ (notice that $\alpha$ depends on $l+1$). Let 

\small\[h^{e}_{d,\pi, \langle g_{i}^{\ast}:i<l+1\rangle}(n)=\max\left\{\begin{array}{l}m^{\ast}, \max\{B_{l+1}^n(j):\exists \xi\leq \alpha (j\leq |_{\xi}d |)\}\\\min\{k:\forall \xi<\alpha\forall i\in{} _{\xi}d \forall \sigma\in \mbox{dom}(_{\xi}\pi^{i}) (\sigma\in k^{i}) \}\end{array}\right\}. \]

\normalsize Now we do the verification.

For (2a), the proof is the same as in the case $l=1$.

For (2) we have multiple cases. The proof is the same as in the case $l=1$ if $G=\overline{f}\in C^{n}_{l+1}$. If  $G=\overline{f}\notin C^{n}_{l+1}$ we know that there is a function in $G$ exceeding $B_{l+1}^n$. Assume that $g(j)>B_{l+1}^n(j)$ and that, for all $i<l+1$, $m<j$, $g_{i}(m)\leq B_{l+1}^n(m)$, i.e., $j$ is the first time that $g$ is above $B_{l+1}^n$. Let $G=G_{0}\cup G_{1}$ be such that for all $f\in G_{0}$, $f(j)>B_{k}^n(j)$ and for all $a\in G_{1}$, $a(j)\leq B_{k}^n(j)$. Again, we have two cases.

If $j\geq |_{\xi}d|$ for all $\xi\leq \alpha$, then we define $G'_{0}=G_{0}\restrict j^{\frown}0^\infty$, i.e., $G'_{0}$ is the collection of functions that look like a function of $G_{0}$ up to $j$ and then is followed by $0$. Now, $G'=G_{0}'\cup G_{1}\in C^{n}_{l+1}$ so there is $\xi \leq \alpha $ such that $\langle _{\xi}d,\, _{\xi}\pi, F_s\cup G' \rangle$ extends $\langle d,\, \pi, F_s\cup G' \rangle$ and this, since $|_{\xi}d|<j$ will be compatible with $\langle _{\xi}d,\, _{\xi}\pi, F_s\cup G \rangle$.

Finally, assume that $j< |_{\xi}d|$ for some $\xi\leq \alpha$. Notice that $|G_{1}|<l+1$, so, thanks to the definition of $B_{l+1}^n$, given  $g\in G_{0}$, we have \[g(j)>B_{l+1}^n(j)\geq h^{e}_{d', \pi, \langle a\restrict j+1: a\in G_{1}\rangle}(n),\] with $d'=d0^{|_{\xi}d|-|d|}$.

By our inductive hypothesis (specifically, by (2a)) we have that $\langle d', \pi, F_{s}\cup \{g\}\rangle$ is compatible with an extension $r\leq\langle d',\pi,F_s\cup G_{1}\rangle$ with $\varphi^{\langle ^{r}d,\, ^{r}\pi\rangle}(n)\downarrow<h^e_{d', \pi, \langle h\restrict j+1: h\in G_{1}\rangle}(n),$ and $r$ does not depend on $g$ (but it depends on $j$, $d'$, $G_1$ and $n$). As a final line, notice that

\[\varphi^{\langle ^{r}d,\, ^{r}\pi\rangle}(n)\downarrow<h^e_{d', \pi, \langle h\restrict j+1: h\in G_{1}\rangle}(n)\leq B^{n}_{|G_{1}|}(|d'|)\leq B^{n}_{l+1}(|_{\xi}d|)\leq h^{e}_{d,\pi, \langle g_{i}^{\ast}:i<l+1\rangle}(n)\]

\end{proof}

\subsection{Evasion Degrees}

\normalsize Now we will look at the results relating evasion degrees to the rest of the nodes in the computable version of Cicho\'{n}'s diagram. 

\begin{theorem}\label{hyperimmuneevade}If $A$ computes a weakly 1-generic, then $A$ is an evasion degree.\end{theorem}

\begin{proof}If $A$ computes a weakly 1-generic, then it computes a function escaping all computably meager sets. Furthermore, the collection of sets predicted by any computable predictor is a computably meager set by Lemma \ref{meager}, and so $A$ computes a function evading any computable predictor.\end{proof}

\begin{theorem}\label{DNCevade}If $A$ is DNC, then $A$ is an evasion degree.\end{theorem}

\begin{proof}Let $\{P_e=\langle D_e,\pi_e\rangle\}$ be a list of the partial computable predictors by index $e$. We note that by a result of Jockusch in \cite{jockusch1989degrees}, $A$ computes a DNC function if and only if it computes a strongly DNC function---that is, a function $f\leq_T A$ such that for all $n$, and $\forall e\leq n\ f(n)\neq\varphi_e(e)$. Then we can define $g(m)=f(n_m)$ for $n_m$ large enough that $f(n_m)\neq \pi_e(g\restrict_m)$ for all $e\leq m$.

To find $n_m$, we use the fact that $g\restrict m$ is a finite set. For each $e\leq m$ there is a index $k_{e}$ such that $\varphi_{k_e}(n)=\pi_{e}(g\restrict m)$ for all $n$.  Given $g\restrict m$ and $\{P_e=\langle D_e,\pi_e\rangle\}$ we can find computably $n_m\geq k_{e}$ for all $e\leq m$.
\end{proof}

\begin{corollary}\label{meagerengulfevade}If $A$ is weakly meager engulfing, then $A$ is an evasion degree. Furthermore, if $A$ is not low for weak 1-generics, then $A$ is an evasion degree.
\end{corollary}

\begin{proof}By a result of Rupprecht in \cite{rupprechtthesis} $A$ is weakly meager engulfing if and only if it is high or DNC. If $A$ is high, then it has hyperimmune degree, and so is an evasion degree by Theorem \ref{hyperimmuneevade} and the fact that hyperimmune degrees compute weakly 1-generics. If $A$ is DNC, then it is an evasion degree by Theorem \ref{DNCevade}. This completes the proof.

Surprisingly, we actually get an even stronger result, which differs greatly from the analogous case on the set theoretic side:

By a result of Stephan and Yu in \cite{stephan2006lowness}, $A$ is not low for weak 1-generics if and only if $A$ is hyperimmune or DNC. Combining this with Theorem \ref{hyperimmuneevade} and Theorem \ref{DNCevade}, we have the desired result.
\end{proof}

\begin{definition}We define a trace to be a function $g:\omega\rightarrow[\omega]^{<\omega}$ with $|g(n)|=n$. A computable trace will simply have $g$ computable.

We define $A\in2^{\omega}$ to be computably traceable if for all $f\in\omega^{\omega}$ with $f\leq_T A$, there is a computable trace $g$ such that $f(n)\in g(n)$ for all $n$.
\end{definition}

\begin{theorem}\label{evadeschnorr}If $A$ is an evasion degree then $A$ is not low for Schnorr tests.
\end{theorem}

\begin{proof}Let $A$ be low for Schnorr tests. Then, by a result of Terwijn and Zambella in \cite{terwijn2001computational}, it follows that $A$ is computably traceable. Let $f\leq_T A$ be a total function. Then we define $g$ by $g(n)=f\restrict_{I_n}$ where $I_n=\left[\frac{n(n-1)}{2},\frac{n(n+1)}{2}\right)$ (any computable partition of $\omega$ into disjoint sets with $|I_n|=n$ works here). Note that since $g\leq_T f\leq_T A$, it follows that $g$ is computably traceable. Then, by assumption, there is a computable trace $T$ where $T(n)\subset\omega^n$, $|T(n)|=n$, and $g\restrict_{I_n}\in T(n)$. However, for any $n$, there are at most $n-1$ values on which a first difference between members of $T(n)$ is witnessed. Put another way, there are at most $n-1$-many values $i$ such that there are $\sigma,\tau\in T(n)$ with $\sigma\restrict_i=\tau\restrict_i$, but $\sigma(i)\neq\tau(i)$. So there must be $j\in I_n$ where for all $\sigma,\tau\in T(n),\sigma\restrict_j=\tau\restrict_j\Rightarrow\sigma(j)=\tau(j)$. Then, we can computably build a predictor which predicts $f$ by adding $j$ to $D$, and accurately predicting all the elements of the trace.\end{proof}

This finishes the positive results (or implication results) involving the evasion degrees. Now, we show the negative results:

\begin{definition}
$X$ is \it{weakly Schnorr engulfing} if and only if $X$ computes a null set that contains all computable reals.
\end{definition}

\begin{corollary}\label{Evasion not engulfing}
There is an evasion degree that is not weakly Schnorr engulfing.
\end{corollary}

\begin{proof}
In \cite{rupprecht2010relativized}, Rupprecht showed that any Schnorr Random that is hyper-immune-free is not weakly Schnorr engulfing.

As any Schnorr Random is either High or DNC, using Theorem \ref{DNCevade}, we have that a Schnorr Random that is hyperimmune-free is an evasion degree.
\end{proof}

To prove the next theorem we will use a modified notion of clumpy trees introduced by Downey and Greenberg in \cite{downey2008turing}.

\begin{definition}A \emph{perfect function tree} is a function $T:2^{<\omega}\rightarrow2^{<\omega}$ that preserves extension and compatibility.

Let $T$ be a perfect function tree, $\sigma\in\mbox{im}\ T$, the image of $T$, and let $I=\langle I_n : n\in \omega \rangle$ where $I_n=[\frac{n(n-1)}{2}, \frac{n(n+1)}{2})$. We say that $T$ \emph{contains an $I$-clump above $\sigma$} if there is $n$ such that $|\sigma|=\frac{n(n-1)}{2}$ and for all binary strings $\tau$ of length $n, \sigma\tau=T(\rho\tau)$,
where $\sigma=T(\rho)$. We further define $T$ to be \emph{$I$-clumpy} if for all $\sigma\in T$ there exist $\tau\in T$ extending $\sigma$ such that $T$ contains an $I$-clump above $\tau$.\end{definition}

\begin{theorem}\label{Weaklynotprediction}There is an $A\in 2^{\omega}$ which is not of evasion degree, but it is  weakly Schnorr engulfing. 
\end{theorem}

\begin{proof}

Let $\{X_d\}_{d\in \omega}$ be a enumeration of all computable reals in $2^{\omega}$.

By an observation of Rupprecht in Theorem 19 of \cite{rupprecht2010relativized}, we have that if there are disjoint intervals $J^{n}$ such that $|J^{n}|\geq n$, the sequence $\langle A\restrict_{J^{n}}:n\in\omega \rangle$ is computable in $A$ and for every $d\in \omega$ there are infinitely many $n$ such that $A\restrict_{J^{n}}=X_{d}\restrict_{J^{n}}$ then $A$ is weakly meager engulfing.

The idea of this proof will be to use forcing with computable trees with some specific properties. First, at the $e$th stage, we will be pruning to a tree consisting entirely of paths $A$ for which $\varphi_e^A$ is computably predictable. We will use this to ensure that the result of our forcing does not compute an evading function. Second, the trees will be clumpy, allowing us to choose extensions which occasionally agree with $X_{d}$. This will mean our resulting set belongs to a weakly Schnorr engulfing degree. 

Given an initial segment $A_{e-1}$ and a computable tree $T_{e-1}$ extending this initial segment, we will prune our tree to $T_e$, so that there is a single predictor that always predicts $\varphi_e^A(n)$ for every remaining path $A\in T_e$ while maintaining the clumpiness requirement. 

At every stage in our construction, we will assume that there is no initial segment $\sigma$ in our current tree $T_{e-1}$ such that $\varphi_e^A$ is non-total for all paths $A\succ\sigma$. Additionally, we will assume that for any $\sigma\in T_e$, there exist $\tau_1,\tau_2\succ\sigma$ such that $\varphi_e^{\tau_1}\neq\varphi_e^{\tau_2}$. If either of these fail, we define $A_e=\sigma$ and $T_e$ is the portion of $T_{e-1}$ extending $\sigma$. In either case, the clumpiness condition is preserved for the next stage. In the case that the first assumption fails, $\varphi_e^A$ is not total for all $A\succ\sigma$, and so we need not predict it accurately. In the case that the latter assumption fails, $\varphi_e^A$ is computable for all $A\succ\sigma$, and so can be predicted easily.

Each run of the construction will go as follows: We will rotate through three distinct goals. We can think of them as clumping, differentiating and predicting. 

First, we will add clumps. Given a collection $\{\sigma_i\}$ of initial segments in the tree, each of length $n$, we will search for $m>n$ such that $T_{e-1}\restrict_m$ contains an $I$-clump above $\sigma_i$ for each $\sigma_i$. Then, the collection given by $T_{e-1}\restrict_m$ will be the $\{\tau_i\}$ for the next stage. 

Next, we will differentiate. We look for $j>m$ so that each $m$-length $\tau_i$ has an extension $\gamma_i$ of length $j$ such that $\varphi_e^{\gamma_i}$ is distinct for each such $\gamma_i$. We are guaranteed to find these by our previous assumption about splitting. 

In the final step, we predict. We now look for $d\in\omega$ such that $\varphi^{\gamma_i}_e(d)$ is undefined for all $\gamma_i$ previously defined. We add this $d$ to $D$ for the predictor we are building, and for each $\gamma_i$ we look for a further extension $\sigma_i\succ\gamma_i$ such that $\varphi_e^{\sigma_i}(k)\downarrow$ for all $k\leq d$. Then we define $\pi(\varphi_e^{\sigma_i}\restrict_d)=\varphi_e^{\sigma_i}(d)$. For all other strings $a$ of length $d$, we can define $\pi(a)=0$. Now, finally, these $\sigma_i$ become the initial segments of the tree that we start with for the next pass through these three steps. We repeat the process indefinitely. 

Finally, once $T_e$ is defined, given that $e=\langle d, s \rangle$ we will pick $A_e\succ A_{e-1}$ with $A_e\restrict I_{m}=X_{d}\restrict I_{m}$ with $m> s$. Such a string is guaranteed to exist because of the clumpiness condition on our tree.

Then, $A=\bigcup A_e$ is the desired degree, as it is a path through each $T_e$, and so $\varphi_e^A$ is computably predictable, but by construction, for every $d\in \omega$ there are infinitely many $n$ such that $A\restrict_{I_n}=X_{d}\restrict_{I_n}$
\end{proof}

\begin{corollary}\label{lowschnorrnotevade}There is a degree which is not computably traceable, but not an evasion degree.\end{corollary}

\begin{proof}This is an immediate result from Theorem \ref{Weaklynotprediction} and the fact that weakly Schnorr engulfing implies not low for Schnorr, i.e., not computably treceable.\end{proof}

In the proof of Theorem 19 of \cite{rupprecht2010relativized}, Rupprecht uses a forcing that is analogue to infinite equal forcing (known as Silver forcing in Set Theory). This forcing uses partial functions from $\omega$ to $2$ with a coinfinite computable domain.

It is possible to understand the proof of Theorem \ref{Weaklynotprediction} in the same terms if we visualize the clumps as the intervals where the function is not defined. Nevertheless, using computable trees instead of partial functions helps when we define the predictor. It is an open question whether or not you can have a degree that is not computable traceable, not weakly meager engulfing and not of evasion degree.

In our finished diagram including prediction and evasion (Figure \ref{computablecichonwithevasion}), we have included some of the alternate characterizations of nodes we used that include properties of and relations to the computable functions.

\section{Rearrangement}

The rearrangement number was recently introduced in \cite{blass2016rearrangement} by Blass, Brendle, Brian, Hamkins, Hardy, and Larson. All results and definitions about this characteristic can be found there.

\subsection{Definitions}

\begin{definition}The \emph{rearrangement number} $\mathfrak{rr}$ is defined as the smallest cardinality of any family $C$ of permutations of $\omega$ such that, for every conditionally convergent series $\sum a_n$ of real numbers, there is a permutation $p\in C$ for which 
\[\sum a_{p(n)}\neq\sum a_n.\]\end{definition}

A priori, there are a few different ways of making this happen, namely making the permuted series diverge to infinity, making the permuted series oscillate, and making the permuted series sum to a different finite sum than the original series. In practice, oscillation is easier to achieve than the other two, and so it only makes sense to isolate the other two possibilities, giving a few additional characteristics, where the variation requirement is stronger.

\begin{definition}We present three additional refinements, giving slightly different characterizations:
\begin{itemize}
\item$\mathfrak{rr}_f$ is defined the same way as $\mathfrak{rr}$, but where the sum is required to converge to a different finite number.
\item$\mathfrak{rr}_i$ is defined the same way, but the sum is required to diverge to either positive or negative infinity.
\item$\mathfrak{rr}_{fi}$ is defined the same way, but the sum is required to either diverge to infinity (positive or negative) or converge to a different finite number.
\end{itemize}\end{definition}

Simply by definition, one can easily see that $\mathfrak{rr}\leq\mathfrak{rr}_{fi}\leq\mathfrak{rr}_f,\mathfrak{rr}_i$. The authors in \cite{blass2016rearrangement} were able to show that it is consistent that $\mathfrak{rr}<\mathfrak{rr}_{fi}$, but were unable to conclusively show whether or not the latter three characteristics were separable from each other. Similarly, on the effective side, we have been unable to separate the finite case, the infinite case, or the case allowing either from each other, and so here we will only present the highness notions analogous to $\mathfrak{rr}$ and $\mathfrak{rr}_{fi}$ (although it should be clear what the other two would look like).

\begin{figure}
\centering
\begin{tikzpicture}
  \matrix (m) [matrix of math nodes,row sep=3em,column sep=1em,minimum width=2em]
  {
     &\mbox{Non}(\mathcal{M})&\mbox{Cof}(\mathcal{M})&&\mbox{Cof}(\mathcal{N})\\
     \mbox{Cov}(\mathcal{N})&\mathfrak{rr}&&\mathfrak{rr}_{fi}&\\
     &\mathfrak{b}&\mathfrak{d}&&\\
     \mbox{Add}(\mathcal{N})&\mbox{Add}(\mathcal{M})&\mbox{Cov}(\mathcal{M})&&\mbox{Non}(\mathcal{N})\\};
  \path[commutative diagrams/.cd, every arrow, every label,font=\scriptsize]
    (m-4-1) edge (m-2-1)
            edge (m-4-2)
    (m-2-1) edge (m-2-2)
    (m-2-2) edge (m-1-2)
    		edge (m-2-4)
    (m-3-2) edge (m-2-2)
    		edge (m-3-3)
    (m-3-3) edge (m-1-3)
    		edge (m-2-4)
    (m-4-2) edge (m-3-2)
    		edge (m-4-3)
    (m-1-2) edge (m-1-3)
    (m-1-3) edge (m-1-5)
    (m-4-3) edge (m-3-3)
    		edge (m-4-5)
    (m-4-5) edge (m-1-5);
\end{tikzpicture}
\caption{Cicho\'{n}'s diagram including $\mathfrak{rr}$ and $\mathfrak{rr}_{fi}$.} \label{cichonwithrearrangement}
\end{figure}
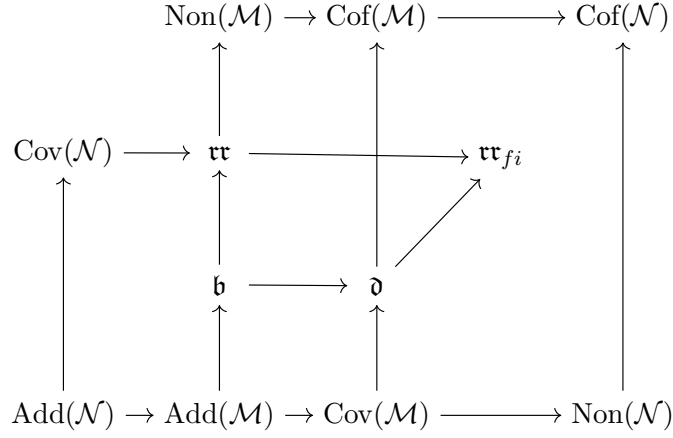

\begin{definition}We define a conditionally convergent series of rationals $\sum a_n$ to be \emph{computably imperturbable} if, for all computable permutations $p$, we have that
\[\sum a_n=\sum a_{p(n)}.\]
Also, we define $\sum a_n$ to be \emph{weakly computably imperturbable} if no computable permutation $p$ has that either
\[\sum a_{p(n)}=B\neq A=\sum a_n\quad\mbox{or}\quad\sum a_{p(n)}=\pm\infty.\]
Equivalently, we can define a series to be weakly computably imperturbable if the only way we get inequality of series under computable permutation is by oscillation, that is
\[\sum a_n\neq\sum a_{p(n)}\Rightarrow\sum a_{p(n)}\mbox{ fails to converges by oscillation}.\]
Finally, we define a real $X\in2^{\omega}$ as \it{(weakly) computably imperturbable}\textnormal{ if it computes a series with the corresponding property. We will refer to (weakly) computable imperturbable just as (weakly) imperturbable.}\end{definition}

We present here known facts about $\mathfrak{rr}$ and $\mathfrak{rr}_{fi}$ along with their computable analogs. All results can be found in \cite{blass2016rearrangement}.

\begin{theorem}\label{rrfacts}The following relationships are known for $\mathfrak{rr}$ and $\mathfrak{rr}_{fi}$.
\small{{\center\begin{tabular}{|l|l|c|}
\hline 
Cardinal Char.&Highness Properties&Theorem\\
\hline 
$\mathfrak{b}\leq\mathfrak{rr}$&high $\Rightarrow$ imperturbable&\ref{highimperturb}\\
\hline
$\mathfrak{d}\leq\mathfrak{rr}_{fi}$&hyperimmune $\Rightarrow$ weakly imperturbable&\ref{hyperimmuneimperturb}\\
\hline
$\mbox{cov}(\mathcal{N})\leq\mathfrak{rr}$&computes a Schnorr Random $\Rightarrow$ imperturbable&\ref{schnorrimperturb}\\
\hline
$\mathfrak{rr}\leq\mbox{non}(\mathcal{M})$&imperturbable $\Rightarrow$ weakly meager engulfing&\ref{meagerengulfimperturb}\\
\hline
$\mbox{CON}(\mbox{cov}(\mathcal{N})<\mathfrak{rr})$&imperturbable $\not\Rightarrow$ computes a Schnorr Random&Open\\
\hline
$\mbox{CON}(\mathfrak{b}<\mathfrak{rr})$&imperturbable $\not\Rightarrow$ high&\ref{imperturbhighsep}\\
\hline
$\mbox{CON}(\mathfrak{rr}<\mathfrak{rr}_{fi})$&weakly imperturbable $\not\Rightarrow$ imperturbable&\ref{imperturbweaksep}\\
\hline
$\mbox{CON}(\mathfrak{d}<\mathfrak{rr}_{fi})$&weakly imperturbable $\not\Rightarrow$ hyperimmune&\ref{imperturbhighsep}\\
\hline
\end{tabular}}}

\

\normalsize These results can be seen in figure \ref{cichonwithrearrangement} and \ref{computablecichonwithrearrangement}.
\end{theorem}

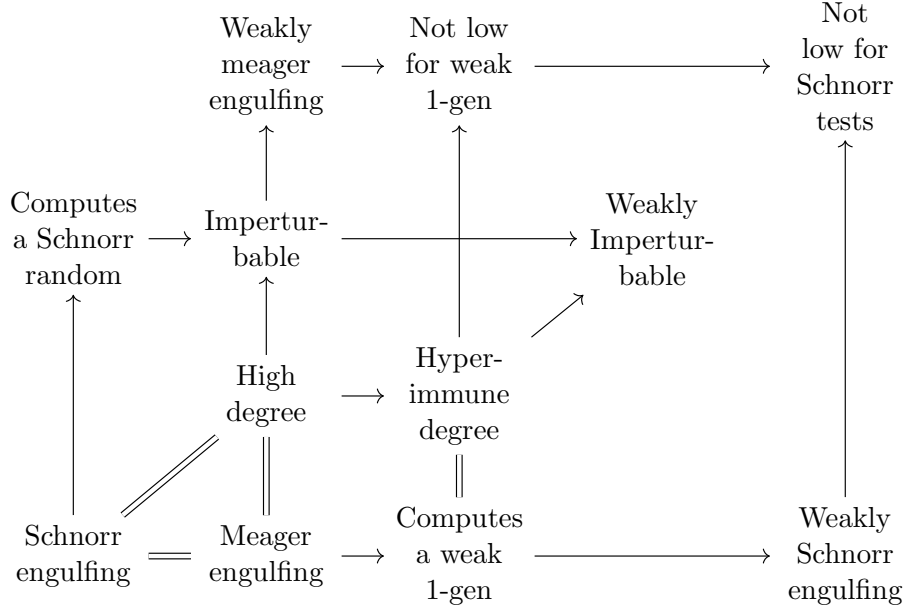
\begin{figure}
\centering
\begin{tikzpicture}[auto,
	every node/.style ={rectangle, fill=white,
      text width=4.5em, text centered,
      minimum height=2em}]
  \matrix (m) [row sep=1.5em,column sep=1.5em]
  {
     &\node (22) {Weakly meager engulfing};&\node (23) {Not low for weak 1-gen};&&\node (25) {Not low for Schnorr tests};\\
     \node (31) {Computes a Schnorr Random};&\node (32) {Impertur-bable};&&\node (34) {Weakly Impertur-bable};&&\\
     &\node (42) {High degree};&\node (43) {Hyper-immune degree};&\\
     \node (51) {Schnorr engulfing};&\node (52) {Meager engulfing};&\node (53) {Computes a weak 1-gen};&&\node (55) {Weakly Schnorr engulfing};\\
     };
	\path[commutative diagrams/.cd, every arrow, every label,font=\scriptsize]
    (31) edge (32)
    (22) edge (23)
    (23) edge (25)
    (32) edge (22)
    	 edge (34)
    (42) edge (32)
         edge (43)
    (53) edge (55)
    (52) edge (53)
    (55) edge (25)
    (51) edge (31)
    (43) edge (34)
         edge (23)
	(51) edge[-,commutative diagrams/equal] (52)
	(51) edge[-,commutative diagrams/equal] (42)
	(42) edge[-,commutative diagrams/equal] (52)
    (53) edge[-,commutative diagrams/equal] (43);
\end{tikzpicture}
\caption{Effective Cicho\'{n}'s diagram including imperturbability.} \label{computablecichonwithrearrangement}
\end{figure}

\subsection{Imperturbability results}

\normalsize The following is an adaptation of Theorems 15 and 16 in \cite{blass2016rearrangement}.

\begin{theorem}\label{highimperturb}If $X$ is high, then it is imperturbable.\end{theorem}

\begin{proof}Let $X\in2^{\omega}$ be high. By a classic result of Martin in \cite{martin1966classes}, this means that there is a (strictly increasing) function $f\leq_T X$ such that $f$ dominates all computable functions. Let $\sum a_n$ be any computable conditionally convergent series. Define the sequence $\{b_k\}$ by
\[b_k=\begin{cases}a_n&k=f^n(0)\\0&\mbox{otherwise}\end{cases},\]
using the convention that $f^n$ is the $n$-times application of $f$, that is 
\[f^n(a)=\overbrace{f(\cdots f(f}^n(a))).\] 
We claim that $\sum b_{p(n)}=\sum a_n$ for all computable permutations $p$. To see that this is true, for each $e\in\omega$, we will define a computable function $g_e$ such that if $\varphi_e$ is a permutation, it follows that $\varphi_e(i)\leq n,g_e(n)\leq \varphi_{e}(j)\Rightarrow i\leq j$ for all $i,j\in\omega$. Clearly, given such computable functions, we can see that the series $\sum b_k$ defined above has the desired property, as $f$ dominates all of the $g_e$, and so no computable permutation alters the order of any more than finitely many non-zero elements, leaving the sum unchanged.

In order to define $g_e(n)$, we first assume $\varphi_e$ is a permutation, if it isn't, nothing that we do matters, as we do not have to defeat it. We begin searching computably for $A_n=\{l\in\omega:\varphi_e(l)\leq n\}$. At some finite stage in our computation, we will have found $l_k$ such that $\varphi_e(l_k)=k$ for all $k\leq n$. This follows from the fact that $\varphi_e$ is a permutation. Then, let $a=\max\{l_k:k\leq n\}$. Finally, we can define $g_e(n)=\max\{\varphi_e(m):m\leq a\}$. This $g_e$ has the desired property by construction. \end{proof}

The following is an adaptation of Theorem 18 in \cite{blass2016rearrangement}.

\begin{theorem}\label{hyperimmuneimperturb}If $X$ is of hyperimmune degree, then $X$ is weakly imperturbable.\end{theorem}

\begin{proof}This proof will be very similar to that of Theorem \ref{highimperturb}. Here, let $X$ be of hyperimmune degree. Then, in particular, there is some $f\leq_T X$ such that $f>\varphi_e$ infinitely often for any $e$. That is, for every $e$, there are infinitely many $n$ with $f(n)>\varphi_e(n)$. Here, we will also require that $f$ is strictly increasing. Again, for $\sum a_n$ some computable conditionally convergent series, we define the sequence $\{b_k\}$ by 
\[b_k=\begin{cases}a_n&k=f^n(0)\\0&\mbox{otherwise}\end{cases}.\]

We claim that for all $\varepsilon>0$ and $e\in \omega$, if $\varphi_e$ is a permutation, then there are infinitely many distinct pairs $i,j\in\omega$ such that 
\[\left|\sum\limits_{k=0}^i b_{\varphi_e(k)}-\sum\limits_{k=0}^j a_k\right|<\varepsilon.\]
To see that this is true, we can use exactly the same $g_e$ as we used in Theorem \ref{highimperturb}. Remember, if $\varphi_e$ is a computable permutation, then $g_e$ is total computable. Since $f$ is not dominated by any computable function, it follows that $f(n)>g_e(n)$ infinitely often. In particular, since $f$ is monotone increasing, there must be infinitely many $n$ so that $f^{n+2}(0)\geq g_e(f^n(0))$. For each such $n$, there is an initial partial sum of the $b_{\varphi_e(k)}$ which differs from $\sum\limits_{k=0}^n a_k$ by at most $|a_{n+1}|$. These pairs have the desired property. 

Now, since $|a_n|\rightarrow0$ for $n$ large, the initial partial sums of the $b_{\varphi_e(k)}$ are infinitely often arbitrarily close to those of the $a_n$. It follows that $\sum b_{\varphi_e(k)}$ can neither converge to a different limit than $\sum a_n$, nor diverge to infinity. Thus we have that $\sum b_k$ is a weakly imperturbable series, as desired.\end{proof}

For the next lemma we will need the following definitions and facts from \cite{rute2012algorithmic}:

\begin{definition}A \emph{computable metric space} is a triple $\mathbb{X} = (X, d, S)$
such that\begin{enumerate}[(1)]
\item $X$ is a complete metric space with metric $d:X\times X\rightarrow[0,\infty)$.
\item $S=\{a_i\}_{i\in\omega}$ is a countable dense subset of $X$.
\item The distance $d(a_i, a_j)$ is computable uniformly from $i$ and $j$.\end{enumerate}
A point $x\in X$ is said to be computable if there is a computable function $h:\omega\rightarrow\omega$ such that for all $m > n$, we have $d(a_{h(m)},a_{h(n)})\leq2^{-n}$ and $x = \lim\limits_{n\rightarrow\infty}a_{h(n)}$. The sequence $(a_{h(m)})$ is the \emph{Cauchy-name for $x$}.\end{definition}

\begin{definition}Let $\mathbb{Y} = (Y, S, d_{\mathbb{Y}})$ be a computable metric space. The space of measurable functions from $(2^{\omega},\lambda)$ to $\mathbb{Y}$, where $\lambda$ is the Lebesgue measure on $2^{\omega}$, is a computable metric space under the metric
\[d_{\mbox{meas}}(f, g) = \int\min(d_{\mathbb{Y}}(f,g),1)\ d\lambda\]
and where the countable dense sets are the test functions of the form
$\varphi(x) = c_i$ when $x\in [\sigma_i
]$ (prefix-free $\sigma_0,\dots, \sigma_{k-1}\in2^{<\omega}; c_0,\dots, c_{k-1} \in S$) and $\varphi(x)=0$ otherwise.
The computable points in this space are called \it{effectively measurable functions}.
\end{definition}

\begin{lemma}[Rute \cite{rute2012algorithmic}]\label{Schnorrconvergence}Suppose $f:(2^{\omega},\lambda)\rightarrow \mathbb{Y}$ is effectively measurable with Cauchy-name $(\varphi_n)$ in $d_{\mbox{meas}}$. The limit $\lim\limits_{n\rightarrow\infty}\varphi_n(x)$ exists on all Schnorr Randoms $x$.\end{lemma}


\begin{lemma}[Kolmogorov\cite{kolmogoroff1928summen}]\label{kolmogorov} Let $X_0,\dots,X_n$ be independent random variables with expected value $E[X_i]=0$ and finite variance. Then for each $\epsilon>0$
\[P\left[\max\limits_{0\leq k\leq n}\left(\sum_{i=0}^k X_i\right)\geq\epsilon\right]\leq\frac{1}{\epsilon^2}\sum\limits_{i=0}^n\mbox{Var}(X_i).\]
\end{lemma}

This collection of lemmas will be used to prove the following result which is an effectivization of a theorem of Rademacher \cite{rademacher1922einige}.

\begin{lemma}\label{effectiverademacher}If the sequence of rationals $\{a_n\}$ is computable with the limit $\sum a_n^2<\infty$ also computable, and $X\in 2^\omega$ is a Schnorr Random, then $\sum a_n(-1)^{X(n)}$ converges.\end{lemma}

\begin{proof}To see this, we will find a Cauchy-name for the function $f(x)=\sum a_n(-1)^{x(n)}$ in the metric $d_{\mbox{meas}}$. Then we need only apply Lemma \ref{Schnorrconvergence} to get the desired result.

Given a computable sequence of rationals $\{a_n\}$ with $\sum a_n^2<\infty$ computable, and $m\in\omega$ we define $\varphi_m(x)=\sum\limits_{n=0}^{i_m}a_n(-1)^{x(n)}$ where $i_m$ is least such that 
\[\sum_{n=i_m}^\infty a_n^2 <\frac{1}{8^{m+1}}.\]
To see that this is a Cauchy-name, given $j>m$, if we define \[A_{j,m}=\left\{x\in2^{\omega}:|\varphi_j(x)-\varphi_m(x)|\leq\frac{1}{2^{m+1}}\right\}\] we have that 
\begin{align*}d_{\mbox{meas}}(\varphi_j,\varphi_m)&\leq\int\limits_{A_{j,m}}|\varphi_j(x)-\varphi_m(x)|\ d\lambda+\int\limits_{2^{\omega}\setminus A_{j,m}}1\ d\lambda\\
&\leq\frac{1}{2^{m+1}}+\lambda\left\{x\in2^{\omega}:\left|\sum_{n=i_m+1}^{i_j}a_n(-1)^{x(n)}\right|>\frac{1}{2^{m+1}}\right\}.\end{align*}
However, we can effectively bound the measure of the set in this inequality by \[\left\{x\in2^{\omega}:\left|\sum_{n=i_m+1}^{i_j}a_n(-1)^{x(n)}\right|>\frac{1}{2^{m+1}}\right\}
\subseteq\]\[\bigcup_{k=0}^{\infty}\left\{x\in2^{\omega}:\left|\sum_{n=i_m}^{i_m+k} a_{n}(-1)^{x(n)}\right|>\frac{1}{2^{m+1}}\right\}.\]
Then, applying Lemma \ref{kolmogorov}, we have 
\begin{align*}\lambda\left(\bigcup_{k=0}^{\infty}\left\{x\in2^{\omega}:\left|\sum_{n=i_m}^{i_m+k} a_{n}(-1)^{x(n)}\right|>\frac{1}{2^{m+1}}\right\}\right)&\leq \frac{1}{(1/2^{m+1})^2}\sum_{j=i_m}^{\infty}a_j^2\\
&<\frac{1}{2^{m+1}},\end{align*}
and so $d_{\mbox{meas}}(\varphi_m,\varphi_j)\leq\frac{1}{2^{m+1}}+\frac{1}{2^{m+1}}=\frac{1}{2^m}$, as desired. Thus, $\varphi_m$ is a Cauchy name, as desired. Then, by Lemma \ref{Schnorrconvergence}, it must converge on all Schnorr Randoms.
\end{proof}

\begin{lemma}[Folklore]\label{cantmessupaschnorr}A computable permutation of a Schnorr Random is Schnorr Random.\end{lemma}

The following is an adaptation of Lemma 7 in \cite{blass2016rearrangement}.

\begin{lemma}\label{permmix}Given a computable permutation $p$, there is a computable permutation $q$ with the property that there are infinitely many $i$ such that $\{q(n):n\leq i\}=\{p(n):n\leq i\}$ and infinitely many $j$ such that the same happens with the identity, i.e.,  $\{q(n):n\leq j\}=\{0,\dots,j\}$.\end{lemma}

\begin{proof}We can essentially just build this. Let $p$ be a computable permutation, then we alternate between conditions. We define $q_0(0)=0$, and then we build $q$ in stages such that the domain of $q_s$ will always be an initial segment of $\omega$. For each $s>0$, we do the following:

If $s$ is odd, we aim to add an $i$ so that $\{q(n):n\leq i\}=\{p(n):n\leq i\}$. To do this, we begin to search computably for $m_k\in\omega$ for $k$ on which $q_{s-1}$ has already been defined such that $p(m_k)=q_{s-1}(k)$ for each $k\in\mbox{dom}(q_{s-1})$. Then we will define $q_s$ up to $\max\{m_k\}$ by simply building a bijection between $\{0,\dots,\max\{m_k\}\}$ and $\{p(0),\dots,p(\max\{m_k\})\}$ picking one element at a time while respecting $q_{s-1}$. This is simple, as the collection is computable, and $q_{s-1}$ is already a bijection with a subset, and so we can simply extend. Then, $\max\{m_k\}$ will be the desired $i$.

If $s$ is even, we aim to add a $j$ so that $\{q(n):n\leq j\}=\{0,\dots,j\}$. This is even more straightforward. The $j$ we choose will be $j=\max(\mbox{range}(q_{s-1}))$, and we can simply build a bijection between the finite, computable, same-size sets, $\{0,\dots,j\}\setminus\mbox{dom}(q_{s-1})$ and $\{0,\dots,j\}\setminus\mbox{range}(q_{s-1})$, in order to extend $q_{s-1}$ to $q_s$.

It is straightforward to see that, from the construction, $q=\bigcup q_s$ is a bijection, and $\mbox{range}(q)=\mbox{dom}(q)=\omega$. Thus, $q$ is a computable permutation, and has the desired property.\end{proof}

Note, this result can actually be extended so that, given any two permutations $p_1,p_2$, there is a permutation $q\leq_T p_1\oplus p_2$ such that there are infinitely many $i,j$ such that $\{q(n):n\leq i\}=\{p_1(n):n\leq i\}$ and $\{q(n):n\leq j\}=\{p_2(n):n\leq j\}$.

The following is an adaptation of Theorem 6 in \cite{blass2016rearrangement}.

\begin{lemma}\label{easyoscillation}If $\sum a_n$ is not imperturbable, then there is a computable permutation $p$ such that $\sum a_{p(n)}$ fails to converge due to oscillation.\end{lemma}

\begin{proof}Let $\sum a_n$ be a series which is not imperturbable. That is, there is a computable permutation $p$ such that 
\[\sum a_n\neq\sum a_{p(n)}.\]
We can assume that $\sum a_{p(n)}=\pm\infty$ or $\sum a_{p(n)}=B\neq A=\sum a_n$, otherwise there is nothing to show. Now let $q$ be as in Lemma \ref{permmix}. This $q$ has the desired property. If $\sum a_{p(n)}=\infty$, then for $i$ as in the lemma, we have that
\[\sum_{n=0}^i a_{q(n)}=\sum_{n=0}^i a_{p(n)},\]
thus we can see that these partial sums grow without bound, but simultaneously, for $j$ as in the lemma, we have that
\[\sum_{n=0}^j a_{q(n)}=\sum_{n=0}^j a_{n},\]
and so these partial sums tend towards $A=\sum a_n$. Thus, the whole series must be non-convergent due to oscillation. A similar argument shows that if $\sum a_{p(n)}=B\neq A$, then there are infinite subsequences of partial sums of $\sum a_{q(n)}$ converging to both $A$ and $B$, which also means that $\sum a_{q(n)}$ must be non-convergent due to oscillation.
\end{proof}

\begin{theorem}\label{schnorrimperturb}If $X$ computes a Schnorr Random, then $X$ is imperturbable.\end{theorem}

\begin{proof}Let $X\in2^{\omega}$ and $A\leq_T X$ be Schnorr Random. Then, we claim that if we define $a_n=\frac{(-1)^{A(n)}}{n}$, the series $\sum a_n$ is imperturbable. To see this, let $p$ be a computable permutation, then $\sum a_{p(n)}$
converges by Lemma \ref{effectiverademacher} and Lemma \ref{cantmessupaschnorr}. Namely, the sequence $\left\{\frac{1}{p(n)}\right\}$ is a computable sequence by construction, \[\sum\left(\frac{1}{p(n)}\right)^2=\sum\frac{1}{n^2}=\frac{\pi^2}{6}\] is computably converging to a computable number, and the set of indices of negative entries of our sequence $\{a_{p(n)}\}$ is Schnorr Random by Lemma \ref{cantmessupaschnorr}. Thus, we can apply Lemma \ref{effectiverademacher}, and so the series converges for all computable permutations. Further, since this series must converge for all computable permutations, it follows from Lemma \ref{easyoscillation} that it must be imperturbable.\end{proof}

The following is an adaptation of Theorem 11 in \cite{blass2016rearrangement}.

\begin{theorem}\label{meagerengulfimperturb}If $X$ is imperturbable, then $X$ is weakly meager engulfing.\end{theorem}

\begin{proof}We will actually show that $X$ is weakly meager engulfing in the space of permutations, but there is a computable bijection between $\omega^{\omega}$ and the space of permutations. Let $X$ be imperturbable, then there is a conditionally convergent imperturbable series $\sum a_n\leq_T X$. We claim that the set of permutations leaving this sum unchanged is contained in an $X$-effectively meager set. In particular, the set of permutations which do not make the sum $+\infty$ is contained in the set
\[E=\bigcup\limits_{k\in\omega}\bigcap\limits_{m\geq k}\left\{p:\sum\limits_{n=0}^ma_{p(n)}\leq k\right\}.\]
Now, we simply observe that the intersection
\[E_k=\bigcap\limits_{m\geq k}\left\{p:\sum\limits_{n=0}^ma_{p(n)}\leq k\right\}\]
is $\Pi^0_1$ in $X$, additionally, it is nowhere dense: given $\pi$ an element of $E_{k}$, $\sigma$ an initial segment of it and $M=\left|\sum_{n=0}^{|\sigma|}a_{p(n)}\right| $, we can find $a_{s_1},..., a_{s_{t}}$ such that $s_{i}>\pi(n)$ for all $n\leq |\sigma|$ and $\sum_{i=0}^{t}a_{s_{i}}> M+k$ since $\{a_{n}\}$ is conditionally convergent. Then, we can extend $\sigma$ to the initial segment of a permutation, call this $\tau$, such that \[\sum_{n=0}^{|\sigma|+t}a_{\tau(n)}=\sum_{n=0}^{|\sigma|}a_{\pi(n)}+\sum_{i=0}^{t}a_{s_{i}}> k. \] 
Any extension of $\tau$ is not in $E_{k}$, therefore, $E_k$ is nowhere dense.

Thus, $E$ is an $X$-effectively meager set of permutations containing all computable permutations, as desired.\end{proof}

We can immediately see that almost all of the forgoing implications are not reversible. This follows from the theorems plus existing known cuts of the computable Cicho\'{n}'s diagram. These cuts are cataloged in \cite{bbtnn} \S 4.2.

\begin{corollary}\label{imperturbhighsep}There is an $X$ which is imperturbable and hyperimmune free (also, not high). In particular, there is an $X$ which is weakly imperturbable and is also hyperimmune-free.\end{corollary}

\begin{proof}This is a direct result of Theorem \ref{schnorrimperturb} plus the fact that there is a Schnorr Random which is hyperimmune free. In fact, there is a low ML-random, which we can see from the low basis theorem plus the existence of a universal ML-test. See e.g. \cite{niesbook} Theorem 1.8.37.\end{proof}

\begin{corollary}\label{imperturbweaksep}There is an $X$ which is weakly imperturbable but not imperturbable.\end{corollary}

\begin{proof}We will use the fact that weakly meager engulfing is equivalent to high or DNC, a proof of which can be found in \cite{kjos2006kolmogorov}. The corollary follows directly from Theorems \ref{hyperimmuneimperturb} and \ref{meagerengulfimperturb} plus the existence of a set of hyperimmune degree which is not weakly meager engulfing. Any nonrecursive low c.e.\ set suffices. Obviously, being of hyperimmune degree means that it is also weakly imperturbable. Additionally, by Arslanov's completeness criterion (\cite{niesbook}, 4.1.11), such a set cannot be DNC, and is not high by definition. Thus, the set is also not weakly meager engulfing.\end{proof}



In figure 6 it is possible to see were imperturbability stands in the effective Cicho\'n's diagram.


\section{Questions}

\begin{question}Is there an $A\in2^{\omega}$ of prediction degree which does not compute a Schnorr Random?\end{question}

\begin{question}\label{evadenotlow1gen}Is there an $A$ which is an evasion degree and low for weak 1-generics?
\end{question}

\begin{question}
Is there an $A$ of a degree that is not computable traceable, not weakly meager engulfing and not of evasion degree?
\end{question}

\begin{question}Is imperturbable equivalent to weakly meager engulfing?\end{question}

\begin{question}Does weakly imperturbable imply any known highness notion?\end{question}

\begin{question}Can we separate the finite case and the infinite case of weakly imperturbable from each other or from the combined notion?\end{question}

\begin{question}Is there an $A\in2^{\omega}$ which is non-computable and not weakly imperturbable?\end{question}

\pagebreak

\bibliographystyle{abbrv}
\bibliography{biblio}
%

\end{document}